\documentclass[11pt]{article}

\usepackage[margin=1in]{geometry}
\usepackage{amsmath,bm}
\usepackage{mathrsfs}
\usepackage{amsthm}
\usepackage{amsfonts}
\usepackage{amssymb}
\usepackage{hyperref}
\usepackage{graphicx}
\usepackage{enumerate} 
\usepackage{enumitem}
\numberwithin{equation}{section} 

\usepackage[comma,sort, numbers]{natbib}

\usepackage{color}

\newtheorem{theorem}{Theorem}[section]
\newtheorem{lemma}[theorem]{Lemma}
\newtheorem{proposition}[theorem]{Proposition}
\newtheorem{corollary}[theorem]{Corollary}

\theoremstyle{definition}
\newtheorem{example}[theorem]{Example}
\newtheorem{definition}[theorem]{Definition}

\newtheorem{remark}[theorem]{Remark}

\newcommand{\Emc}{{\mathcal{E}}}

\def\E{{\mathbb E}}

\def\R{{\mathbb R}}

\def\PP{{\mathbb P}}

\def\P{{\mathcal P}}

\def\M{{\mathcal M}}

\def\F{{\mathcal F}}

\newcommand{\te}{\text}

\newcommand{\defeq}{:=}

\newcommand{\wsp}{\mathcal{W}}

\newcommand{\mpt}{\mathcal{F}}

\newcommand{\lyr}[1]{^{(#1)}}

\newcommand{\intW}{L}

\usepackage{mathabx}

\allowdisplaybreaks

\title{ Multiplexons: Limits of Multiplex Networks }

\author{Ankan Ganguly\thanks{Department Mathematics and Statistics, Boston University, \texttt{ankang@bu.edu}} \and Bhaswar B. Bhattacharya\thanks{Department of Statistics and Data Science, University of Pennsylvania, \texttt{bhaswar@wharton.upenn.edu}} }
\date{}

\begin{document}

\maketitle
\begin{abstract}  
In a multiplex network, a set of nodes is connected by different types of interactions, each represented as a separate layer within the network. Multiplexes have emerged as a key instrument for modeling large-scale complex systems, due to the widespread coexistence of diverse interactions in social, industrial, and biological domains. This motivates the development of a rigorous and readily applicable framework for studying  properties of large multiplex networks. In this article, we provide a self-contained introduction to the limit theory of dense multiplex networks, analogous to the theory of graphons (limit theory of dense graphs). As applications, we derive limiting analogues of commonly used multiplex features, such as degree distributions and clustering coefficients. We also present a range of illustrative examples, including correlated versions of Erd\H{o}s-R\'enyi and inhomogeneous random graph models and dynamic networks. Finally, we discuss how multiplex networks fit within the broader framework of decorated graphs, and how the convergence results can be recovered from the limit theory of decorated graphs. Several future directions are outlined for further developing the multiplex limit theory. 
\end{abstract}


\section{ Introduction }
\label{intro}

Network data arises across nearly all areas of science and technology, and network science has evolved into a highly interdisciplinary field that integrates ideas from physics, mathematics, statistics, biology, engineering, and computer science. Complex systems, however, are rarely formed by single, isolated networks. They often comprise multiple interacting and interdependent networks, each capturing different types of relational information. This has led to the development of multilayer networks, where nodes are connected through different types of interactions, organized into multiple layers. With applications appearing across a wide range of scientific disciplines, the study of multilayer networks has grown rapidly in recent years (see the monographs \cite{bianconi2018multilayer,networksbiology,dickison2016multilayer,multiplexnetworkproperties} and the review articles \cite{kivela2014multilayer,battiston2017new,lee2015towards,de2023more} for detailed expositions on this emerging field).  

An important special case of multilayer networks is {\it multiplex networks}, where a common set of nodes is connected by edges representing different types of interactions. The literature on such networks is extensive and continually expanding. The following are a few popular examples: 

\begin{itemize} 

\item {\it Social Networks}: The concept of multiplicity was originally proposed in the context of social networks to model different types of social interactions \cite{wasserman1994social,fienberg1985statistical,pattison1999logit}. Here, multiplicity (the different layers) can correspond to types of relationship (such as friendship, collaboration, or family ties), modes of communication (such as email, phone, or message), among others. A historical example in this direction is the network between Florentine families during the fifteenth century \cite{florentinegraph}. In this example, the nodes (which correspond to influential families in Florence during the fifteenth century) are connected based on marriage (layer 1) and business (layer 2) alliances, resulting in a multiplex network with two layers (see Figure \ref{fig:h}). In recent years, multiplex networks have emerged as an important modeling paradigm for social systems, driven by the proliferation of large-scale online social network data, where the coexistence of multiple types of interactions is a recurring theme (see \cite[Section 4.3.1]{bianconi2018multilayer} and the references therein).
\begin{figure}
    \centering
      \vspace{-0.35in}
    \includegraphics[scale=0.5]{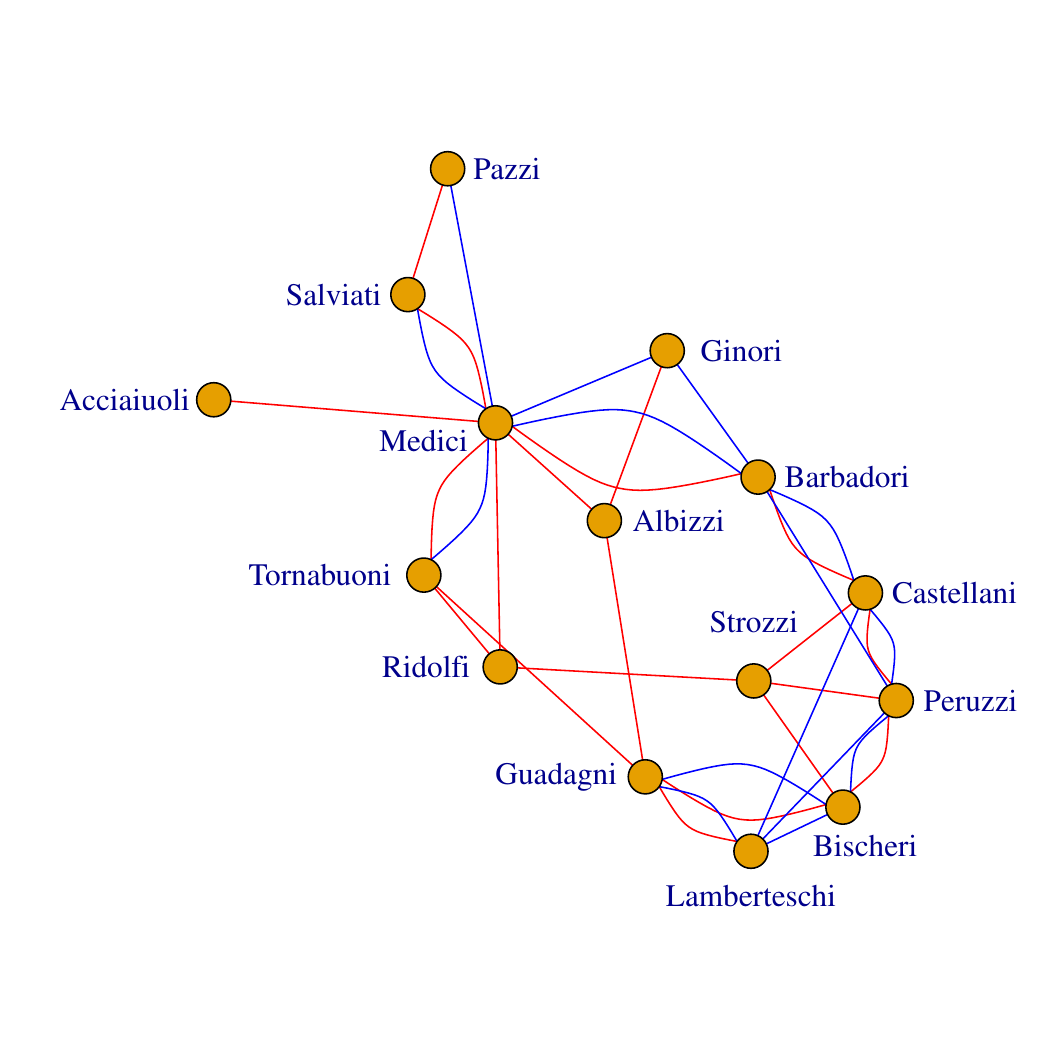}
    \vspace{-0.35in}
    \caption{\small{ Marriage and business alliances between Florentine families of the fifteenth century, shown with red and blue edges, respectively. The data is from \cite{florentinegraph} (see also \cite[Figure 4.2]{bianconi2018multilayer}). }} 
    \label{fig:h}
\end{figure} 

\item {\it Transportation networks}: These are multilayered networked systems, where the layers represent different modes of transportation (also known as multi-mode graphs \cite{kirkpatrick2025shortest}). For example, the transportation network of a major city can be modeled as a multiplex network, where nodes correspond to city locations that are connected across different layers based on modes of transportation, such as bus, metro, or train connections. A prominent dataset in this direction is the European Multiplex Air Transportation Network \cite{cardillo2013emergence}. This is the network of flight connections between European airports, where the layers correspond to flights from different airline companies. Other similar examples are studied in \cite{strano2015multiplex,bergermann2021multiplex,multiplexindex}.

\item {\it Biological networks}: The framework of multilayer networks is also very useful for modeling biological interactions between components of a living system. The advent of network medicine \cite{barabasi2011network} has shown that, to develop a holistic understanding of an organism, it is essential to model the different types of biological interactions together. A particular example that has been extensively studied recently is the connectome (neural connections of the nervous system) of the worm {\it C.~elegans}. The brain of {\it C.~elegans} consists of 302 neurons, which are connected by two types of links: synaptic (chemical) connections and gap junctions (electrical connections), resulting in a multiplex network with two layers (see \cite{bentley2016multilayer} and \cite[Section 4.7]{bianconi2018multilayer} for more details).  

\end{itemize}
Other examples where multiplex networks commonly appear include citation networks, trade and finance multiplex networks, and ecological networks (see \cite[Chapter 4]{bianconi2018multilayer} for relevant references and additional examples).

Various probabilistic models and methods of estimation for multiplex networks have been proposed in recent years. One of the basic models of a random multiplex, which initially appeared in a different guise in the context of network de-anonymization \cite{pedarsani2011privacy}, is the correlated Erd\H{o}s-R\'enyi model. In this model, each layer marginally follows the classical (single-layer) Erd\H{o}s-R\'enyi model, while the edges across the layers are correlated (see Example \ref{example:random12} and Remark \ref{remark:Gnpparameters} for the formal definition). Recently, this has been adopted as the canonical model in random graph matching and related problems (see \cite{mao2024testing,mao2023random,lyzinski2015graph,lyzinski2014seeded} and among several others). Going beyond the basic model, one can consider latent space multiplex models, which extend graphon-based exchangeable random graph models \cite{bickel2011method,bickel2009nonparametric,Lov12,graphlimitsexhangeable2007} to the multiplex setting (see \cite{chandna2022edge,arroyo2021inference,lunagomez2021modeling,macdonald2022latent,skeja2024quantifying,dufour2024inference} and the references therein). As exchangeable models for multiplex networks gain prominence, a natural question emerges: {\it Is it possible to develop a limit theory for multiplex networks that parallels the convergence of dense graphs to graphons?}

In this paper, we provide a self-contained introduction to the limit theory for sequences of dense multiplex networks. To describe the results, we begin by formally defining a multiplex network. 

\begin{definition}\label{defn:G}
For an integer $r \geq 1$, we denote an $r$-multiplex by $\bm{G} = (V(\bm{G}), G_1, G_2, \ldots, G_r)$, where $V(\bm{G})$ is the common vertex set and  $G_s = (V(\bm G), E(G_s))$ is a simple graph with vertex $V(\bm G)$ and edge set $E(G_s)$, for $1 \leq s \leq r$.  Often, slightly abusing notation, we will write $\bm G = (G_1, G_2, \ldots, G_r)$ whenever the common vertex set is clear from the context.  
\end{definition}

Multiplex networks fall within the purview of the more general framework of decorated graphs~\cite{lovasz2010limits} (see also \cite[Chapter 17]{Lov12}), where the edges carry additional information or `decorations'. Specifically, given a set $\bm{Z}$, a $\bm{Z}$-decoration of a simple graph $F= (V(F), E(F))$ is a map $g: E(F) \rightarrow \bm{Z}$. This framework extends traditional binary graphs (where $ \bm{Z} = \{0, 1\}$) to include additional structures such as colored graphs, weighted graphs, and, in particular, multiplex networks.
To represent an $r$-multiplex $\bm{G} = (V(\bm{G}), G_1, G_2, \ldots, G_r)$ as a decorated graph, suppose $\bm{Z}$ is the power set (the collection of all subsets) of $[r]=\{1, 2, \ldots, r\}$ and consider the $\bm{Z}$-decoration of the complete graph with vertex set $V(\bm{G})$ as follows: For $\{i, j\} \in V(\bm{G})$, define $$g((i, j)) = \{ s \in [r]: (i, j) \in E(G_s)\},$$ 
which is the set of layers that contains the edge $(i, j)$. Lov\'asz and Szegedy \cite{lovasz2010limits} developed the left-convergence theory (convergence in terms of homomorphism densities) for decorated graphs, and metric properties of this convergence when $\bm{Z}$ is a finite set were studied in \citet{falgas2016multicolour}. Recently, \citet{abraham2025probability} (see also \citet{zucal2024probabilitygraphons}) introduced the notion of probability graphons, a general framework in which the decorations can take values in an arbitrary Polish space (see Section \ref{sec:probabilitygraphons} for details). Taken together, these results provide a limit theory for decorated graphs, and hence that of multiplex networks, that parallels the convergence of dense graphs to graphons. However, in certain cases, it may not be easy to extract the desired form of a result in the multiplex setting directly from the general theory. Nonetheless, the growing popularity of multiplex networks makes it increasingly important to develop a framework for studying their asymptotic properties that is readily accessible. In this article, we distill the key ideas regarding the convergence of multiplex networks from the general theory of decorated and probability graphons, and compile a comprehensive set of results on multiplex convergence. We complement this with some new results and various examples that highlight the distinctive structural features of multiplex networks. The following is a summary of the results:

\begin{itemize} 

\item {\it Left convergence and decomposition of homomorphism densities}: The notion of left-convergence of a sequence of multiplexes can be naturally defined in terms of convergence of multiplex homomorphism densities (Definition \ref{defn:multigraphlimit}). The more delicate issue is to identify the limiting object. Although an $r$-multiplex has $r$-layers, each of which can be marginally represented by an empirical graphon, to capture the interlayer dependencies in the limit, one needs $2^r-1$ layers (one for each non-empty subset of $[r]$). This is due to the fact that $2^r-1$ parameters are required to characterize an $r$-dimensional Bernoulli distribution \cite{teugels1990multivariatebernoulli,chandna2022edge} (see also Example \ref{example:random12}). 
Hence, to describe the limit of a sequence of $r$-multiplexes, one needs a vector of graphons of length $2^r-1$. We refer to such an object as a {\it multiplexon}.\footnote{A vector of graphons has been referred to by different names in the literature. In the context of multigraph limits \cite{rath2012multigraph} and dynamic graphs \cite{rollin2023dense}, it is called a {multigraphon}. It also appears as a special case of decorated graphons or probability-graphons (see Section \ref{sec:probabilitygraphons}). Here, we introduce the term multiplexon to highlight some of the unique features that arise in the multiplex setting.} 
We also discuss two different ways of decomposing a multiplex/multiplexon, based on their layer structure. These decompositions lead to convenient representations of the homomorphism densities (Proposition \ref{proposition:gtow}) that enable us to identify an appropriate limiting space of multiplexons.

\item {\it Compactness}: A fundamental result in graph limit theory is that the space of graphons is compact under the cut distance, after identifying graphons that have zero cut distance. In Section \ref{sec:compactness}, the analogous results for the multiplexons are established. Specifically, the space of multiplexons equipped with a natural notion of cut distance is compact, after a similar identification.

\item {\it Sampling convergence}: A multiplexon can be used to generate an inhomogeneous random multiplex, which, in particular, includes the latent space multiplex models mentioned above. In Section \ref{sec:randommultiplex}, it is shown that the sampled random multiplex converges to the corresponding multiplexon, both in terms of homomorphism densities and in the cut distance.

\item {\it Counting lemmas}: Similar to graphons, multiplexons also satisfy a counting lemma and an inverse counting lemma (Section \ref{sec:tHW}). Specifically, if two multiplexons are close in the cut distance, then their homomorphism densities are close (counting lemma). Conversely, if two multiplexons are close in terms of their homomorphism densities, then they are also close in terms of their cut distance (inverse counting lemma). 

\item {\it Equivalence of convergence}: The above results combined imply that left-convergence of multiplexes and convergence in terms of the cut distance are equivalent (Section \ref{left}).

\end{itemize}

The above results collectively establish a limit theory for multiplex networks that parallels the theory of graph limits. This theory can be invoked to derive limits of various network features for any sequence of converging multiplexes. As an illustration, in Section \ref{sec:degreeclusteringcoefficient}, we compute the limiting counterparts of different notions of degrees and clustering coefficients in multiplex networks. In Section \ref{sec:examples}, we present several examples of multiplex networks and compute their corresponding multiplexon limits. This includes the Correlated Stochastic Block Model (CSBM), multiplex analogues of threshold graphs and uniform attachment graphs, and dynamic graphs. In Section \ref{sec:probabilitygraphons}, we introduce the notions of decorated graphons and probability graphons, and specifically highlight how the aforementioned results can also be obtained from the general theory. 
Several natural follow-up questions and directions for future research are listed in Section \ref{sec:directions}. We begin by reviewing the essential concepts from graph limit theory in Section \ref{sec:W}. 

\section{Fundamentals of Graph Limit Theory} 
\label{sec:W}

The theory of graph limits provides an analytic framework for studying properties of large graphs. It provides a convenient mathematical language for analyzing a wide range of problems in combinatorics, statistical physics, network science, probability, and statistics. In this section, we will review the basic notions about the convergence of graph sequences. For a comprehensive treatment of the subject, see the monograph by Lov\'asz \cite{Lov12}. 

Given two simple graphs $H = (V(H), E(H))$ and $G = (V(G), E(G))$, denote by $\mathrm{hom}(H, G)$ the set of all homomorphisms of $H$ into $G$. More formally, 
\[\mathrm{hom}(H, G) := \{\phi: V(H) \to V(G):  (\phi(u),\phi(v)) \in E(G) \text{ whenever } (u,v) \in E(H) \} , \]
The {\it homomorphism density} of $H$ into $G$ is defined as: 
$$t(H, G) :=\frac{|\mathrm{hom}(H,G)|}{|V (G)|^{|V (H)|}} . $$
A related quantity of interest is the  {\it injective homomorphism density} of $H$ into $G$: 
\begin{align}\label{eq:injectivetHG}
t_{\text{inj}}(H, G) :=\frac{|\te{injhom}(H,G)|}{|V (G)| (|V(G)| - 1) \ldots (|V(G)| - |V(H)| + 1)} ,  
\end{align} 
 where $\te{injhom}(H,G) := \{ \phi \in \mathrm{hom}(H,G) : \phi \text{ is injective} \}$ is the collection of injective homomorphisms from $H$ to $G$. Note that $t(H, G)$ (respectively, $t_{\text{inj}}(H, G)$) is the probability that a random mapping $\phi: V (H) \rightarrow V (G)$ defines a graph homomorphism (respectively, an injective homomorphism).

A fundamental concept in graph limit theory is left-convergence, which is defined in terms of the convergence of homomorphism densities. 

\begin{definition}\label{defn:graphlimit}\cite{LovSze06,Lov12} A sequence of graphs $\{G_n\}_{n\geq 1}$ is said to be {\it left-convergent} if for every simple graph $H$, $\lim_{n\rightarrow \infty}t(H, G_n)$ exists. 
\end{definition}


To identify the limits of homomorphism densities, consider the space $\wsp$ of all symmetric measurable functions $W: [0, 1]^2 \rightarrow [0, 1]$. (Here, symmetric means $W(x, y) = W(y,x)$, for all $x, y \in [0, 1]$.) The elements of $\wsp$ are called {\it graphons}. For a simple graph $H$ with $V (H)= \{1, 2, \ldots, |V(H)|\}$, define the homomorphism density of $H$ into a graphon $W \in \wsp$ as: 
\begin{align*} 
t(H,W) =\int_{[0,1]^{|V(H)|}}\prod_{(i,j)\in E(H)}W(x_i,x_j) \mathrm dx_1\mathrm dx_2\cdots \mathrm dx_{|V(H)|}.
\end{align*}  
A key result in graph limit theory is that if a sequence of graphs $\{G_n\}_{n\geq 1}$ is left-convergent, then there exists a $W \in \wsp$ such that $\lim_{n\rightarrow \infty}t(H, G_n) = t(H, W)$ (see \cite[Section 11.3.1]{Lov12}). 
In this case, we say $\{G_n\}_{n\geq 1}$ {\it left-converges} to the graphon $W$. 
Note that finite graphs can be represented as a graphon in a natural way as follows: Given a finite graph $G = (V(G), E(G))$, define the {\it empirical graphon} corresponding to $G$ as: 
$$W^G(x, y) =\boldsymbol 1\{(\lceil |V(G)| x \rceil, \lceil |V(G)| y\rceil)\in E(G)\}.$$
In other words, we partition $[0, 1]^2$ into $|V(G)|^2$ squares of side length $1/|V(G)|$, and let $W^G(x, y)=1$ in the $(i, j)$-th square if $(i, j)\in E(G)$, and 0 otherwise. Observe that $t(H, W^{G}) = t(H,G)$ for every simple
graph $H$. 

To obtain a metric that captures the notion of left-convergence, define the {\it cut norm} of a graphon $W\in\wsp$ as:
\begin{align}\label{eq:fgST}
\|W \|_{\square} :=\sup_{S,T\subset [0,1]}\left|\int_{S\times T} W(x,y) \mathrm dx\mathrm dy\right| = \sup_{ f, g: [0, 1] \rightarrow [0, 1] } \left| \int_{[0,1]^{2}} W(x,y) f(x) g(y) \,\mathrm dx\,\mathrm dy \right| ,   
\end{align} 
where the second equality follows from \cite[Lemma 8.10]{Lov12}. The cut norm is equivalent to the more traditional $L_\infty \rightarrow L_1$ operator norm: 
\begin{equation}
\label{sqinfmet}
\| W \|_{\infty \rightarrow 1} = \sup_{ \| f \|_{\infty}, \| g \|_{\infty} \leq 1} \int_{[0,1]^{2}} W(x,y) f(x) g(y) \,\mathrm dx\,\mathrm dy.
\end{equation}
In particular, \cite[Lemma 8.11]{Lov12} shows, $\| W \|_{\square} \leq \| W \|_{\infty \rightarrow 1}  \leq 4  \| W \|_{\square}$. 
To obtain a metric that is invariant to labeling, one defines the  {\it cut distance} between two graphons $U, W \in \wsp$ as follows: 
\begin{align*}
\delta_\square(U, W):=\inf_{\sigma \in \mpt} \| U - W^{\sigma}\|_{\square} , 
\end{align*}  
where $\mpt$ is the set of measure-preserving bijections from $[0,1]$ to $[0,1]$ and $W^{\sigma}(x, y):=W(\sigma(x), \sigma(y))$, for $\sigma \in \mpt$. To obtain a metric space, define an equivalence relation on $\wsp$ as follows: $U\sim W$ when $\delta_\square(U, W) = 0$. Equivalently, $U \sim W$, if there exists measure preserving functions $\sigma_1, \sigma_2 :[0,1]\mapsto [0,1] $ such that $U^{\sigma_1}=W^{\sigma_2}$ (see \cite[Section 10.7]{Lov12}). Denote by ${\wsp}_{\square}$ the space of equivalence classes of graphons with respect to $\sim$.\footnote{The standard notation in the literature for the graphon space, obtained by identifying graphons with cut distance zero, is $\tilde{\wsp}$. Here, we adopt a slightly non-standard notation to avoid confusion with subsequent notations used for various decompositions of a multiplexon.} This is a metric space under the cut distance $\delta_\square$. A fundamental result in graph limit theory is that the space ${\wsp}_{\square}$ is compact (see \cite[Section 9.3]{Lov12}). Further, left-convergence of graphs and convergence in terms of the cut distance are equivalent. Specifically, a sequence of graphs $\{G_n\}_{n\geq 1}$ is left-convergent if and only if it is a Cauchy sequence in the $\delta_\square$ metric. Moreover, a sequence of graphs $\{G_n\}_{n\geq 1}$ is left-converges to a graphon $W$ if and only if 
$\delta_\square(W^{G_n}, W)\rightarrow 0$ (see \cite[Section 11.3.1]{Lov12}). 

\section{Multiplex Networks and Multiplexons }

In this section, we introduce the notion of homomorphism densities for multiplex networks and the concept of cut distance for multiplexons. Next, we discuss left-convergence, present two different ways to decompose a multiplex/multiplexon, and define the empirical multiplexon. Finally, we show how decomposing the source and target multiplexes differently leads to a convenient representation of the homomorphism density.

\subsection{ Homomorphism Densities } 
\label{sec:multiplexGH}

The notion of homomorphism densities extends to the multiplex setting in a natural way.  

\begin{definition}\label{definition:homomorphismGH}
For two  $r$-multiplexes 
$$\bm H = (V(\bm{H}), H_1, H_2, \ldots, H_r) \text{ and } \bm G = (V(\bm{G}), G_1, G_2, \ldots, G_r),$$ a map $\phi: V(\bm{H}) \to V(\bm{G})$ is said to be a homomorphism, if $\phi$ is a homomorphism from $H_s$ to $G_s$, for each $1 \leq s\leq r$. The {\it homomorphism density} from $\bm H$ to $\bm G$ is defined as: 
\[t(\bm{H},\bm{G}) = \frac{|\mathrm{hom}(\bm{H},\bm{G})|}{|V(\bm{G})|^{|V(\bm{H})|}}, \] 
where $\mathrm{hom}(\bm{H},\bm{G}) := \bigcap_{s =1}^r \mathrm{hom}(H_s,G_s)$ 
is the collection of all homomorphisms from $\bm{H}$ to $\bm{G}$. Also, similar to \eqref{eq:injectivetHG}, we define {\it injective homomorphism density} of $\bm{H}$ into $\bm{G}$ as: 
\begin{align}\label{eq:injectivetHGlayers}
t_{\text{inj}}(\bm{H}, \bm{G}) :=\frac{|\te{injhom}(\bm{H}, \bm{G})|}{|V (\bm{G})| (|V(\bm{G})| - 1) \ldots (|V(\bm{G})| - |V(\bm{H})| + 1)} , 
\end{align} 
 where $\te{injhom}(H,G) := \bigcap_{s =1}^r \te{injhom}(H_s,G_s)$ is the collection of injective homomorphisms from $\bm{H}$ to $\bm{G}$. 
\end{definition} 

\begin{remark} As in the case of graphs, $t(\bm{H},\bm{G})$ (respectively, $t_{\text{inj}}(\bm{H}, \bm{G})$) can be interpreted as the probability that a random mapping from $\phi: V(\bm{H}) \to V(\bm{G})$ defines a homomorphism (respectively, an injective homomorphism). Note that a map $\phi: V(\bm{H}) \to V(\bm{G})$ is non-injective if there is a pair of vertices in $V(\bm{H})$ which are mapped to the same vertex in $V(\bm{G})$ under $\phi$. Since the probability of such a collision is $\frac{1}{ |V(\bm{G})| }$, we have 
\begin{align}\label{eq:tHGnoninjective}
|t(\bm{H},\bm{G}) - t_{\text{inj}}(\bm{H}, \bm{G})| \leq \frac{1}{|V(\bm{G})| }\binom{|V(\bm{H})|}{2}. 
\end{align} 
When $\bm{H}$ is a fixed multiplex, the RHS above tends to zero as $|V(\bm{G})| \rightarrow \infty$. This implies that all but a negligible fraction of homomorphisms correspond to injective homomorphisms. 
\end{remark}

To define the continuum analogue of $r$-multiplexes, denote by $\wsp^{r}$ the Cartesian product of $r$ copies of $\wsp$. The elements of $\wsp^{r}$ will be referred to as \emph{$r$-multiplexons}. Given an $r$-multiplexon $\bm W = (W_1, W_2, \ldots, W_r) \in \wsp^{r}$, where $W_s \in \wsp$ is a graphon, for each $1 \leq s \leq r$, and a finite $r$-multiplex 
$\bm H = (V(\bm{H}), H_1, H_2, \ldots, H_r)$, the homomorphism density of $\bm H$ in $\bm W$ is defined as: 
\begin{align}\label{eq:tHW} 
t(\bm{H},\bm{W}) = \int_{[0,1]^{|V(\bm H)|}} \prod_{s =1}^{r} \prod_{(i,j) \in E(H_s)} W_s(x_i,x_j)\, \mathrm dx_1\, \mathrm dx_2\,\dots\, \mathrm dx_{|V(\bm H)|}. 
\end{align}

\subsection{ Cut Distance } 
\label{sec:distancelayers}

The cut norm \eqref{eq:fgST} extends to $\wsp^{r}$ as follows: For $\bm W = (W_1, W_2, \ldots, W_r) \in \wsp^{r}$, 
\begin{align}\label{eq:UWsquarelayer}
\| \bm W \|_{\square} \defeq \sum_{s = 1 }^r \| W_s \|_\square. 
\end{align}
The  {\it cut distance} between two multiplexons $\bm U, \bm W \in \wsp^{r}$ is defined as: 
\begin{align}\label{eq:UWdeltalayer}
\delta_\square(\bm U, \bm W) := \inf_{\sigma \in \mpt} \| \bm U - \bm W^{\sigma} \| _\square , \end{align} 
where $\mpt$ is the set of all measure preserving bijections from $[0,1]\mapsto [0,1]$ and $\bm W^{\sigma} = (W_1^\sigma, W_2^\sigma, $ $\ldots, W_r^\sigma)$, for $\sigma \in \mpt$. Then to obtain a metric space, define an equivalence relation on $\wsp^{r}$ as follows: $\bm U \sim \bm W$ when $\delta_\square(\bm U, \bm W) = 0$ (see Section \ref{sec:deltaUW} for equivalent conditions). Denote by $({\wsp}^r)_{\square}$ the space of equivalence classes of $r$-multiplexons with respect to $\sim$. This is a metric space under the cut distance $\delta_\square$ defined in \eqref{eq:UWdeltalayer}.

\begin{remark} Note that the space $({\wsp}^{r})_\square$ defined above is different from the space $(\wsp_{\square})^{r}$. The latter is the Cartesian product of $r$ copies of ${\wsp}_{\square}$, where one is allowed to transform each of the coordinate graphons with different measure-preserving transformations. This space is not appropriate for capturing limits of multiplexes. This is because the different layers of a multiplex network have the same set of nodes; hence, one needs to apply a single measure-preserving transformation to all the coordinates to preserve the invariance of the homomorphism densities. 
\end{remark}

\subsection{ Left-Convergence } 

 Drawing parallels from Definition \ref{defn:graphlimit}, the natural way to define left-convergence for a sequence of multiplexes is as follows: 

\begin{definition}\label{defn:multigraphlimit} A sequence of $r$-multiples $\{\bm G_n\}_{n\geq 1}$ is said to be {\it left-convergent} if for every $r$-multiplex $\bm H$, $\lim_{n\rightarrow \infty}t( \bm H, \bm G_n)$ exists. 
\end{definition}

While this extension from graphs to multiplexes is natural, the challenge is in identifying the limiting multiplexon. The natural initial guess would be to consider the space of $r$-multiplexons $\wsp^r$ defined in Section \ref{sec:multiplexGH}. To understand why this space is not rich enough, consider the following example:

\begin{example}[Correlated random graphs]\label{example:random12} The correlated Erd\H{o}s-R\'enyi 2-multiplex $$\bm G_n = ([n], (G_n)_{1}, (G_n)_{2})$$ is a random multiplex with vertex set $[n] = \{ 1, 2, \ldots, n \}$ and two layers $(G_n)_{1}$ and $(G_n)_{2}$ such that, independently for $1 \leq i < j \leq n$, 
$$\mathbb{P}((i, j) \in E((G_n)_{1})) = p_1, ~\mathbb{P}((i, j) \in E((G_n)_{2})) = p_2, \text{ and } \mathbb{P}((i, j) \in E((G_n)_{1}) \cap E((G_n)_{2})) = p_{12}, $$
where $(p_1, p_2, p_{12}) \in [0, 1]^3$ such that $p_1 + p_2 - 1 \leq p_{12} \leq \min\{p_1, p_2\}$. 
In other words, in each layer the graphs $(G_n)_{1}$ and $(G_n)_{2}$ follow Erd\H{o}s-R\'enyi random graphs $G(n, p_1)$ and $G(n, p_2)$, respectively, which are correlated across the layers through the joint probability $p_{12}$. As mentioned in the introduction, this model emerged from the study of network privacy \cite{pedarsani2011privacy} and is the basic underlying model in graph matching problems (see \cite{mao2024testing,mao2023random,lyzinski2015graph,lyzinski2014seeded} and the references therein). If one were to describe the limit of $\bm G_n$ as a 2-multiplexon $\bm W = (W_1, W_2) \in \mathcal{W}^2$, it is natural to expect that the graphs in each layer converge to the corresponding marginal graphons. This would enforce $W_1=p_1$ and $W_2=p_2$ almost everywhere. However, this clearly does not ensure convergence of the multiplex homomorphism densities. For example, suppose $H_1 = H_2 = K_2$ is the single edge (the complete graph on 2 vertices) on the vertex set $\{1, 2\}$ and $\bm H = (H_1, H_2)$. Then 
$$t( \bm H, \bm W) = \int_{[0, 1]^2} W_1(x, y) W_2(x, y) \mathrm d x \mathrm d y =p_1p_2.$$
However, $$t(\bm H, \bm G_n) = \frac{1}{n^2}\sum_{1 \leq i \ne j \leq n} \bm 1 \{ (i, j) \in E((G_n)_{1}) \cap E((G_n)_{2}) \} \rightarrow p_{12}.$$ 
in probability. This illustrates the need to consider higher (more than 2)  dimensional product spaces for describing the convergence of 2-multiplexes. 
\end{example}

It is also worth noting that, unlike in the case of graphons, the map $\bm W \rightarrow t( \bm H, \bm W)$ is not continuous in general. To see this, let $r=2$  and consider $\bm H = (H_1, H_2)$ as in the example above. In this case, for $\bm W = (W_1, W_2)$, 
$$t( \bm H, \bm W) = \int_{[0, 1]^2} W_1(x, y) W_2(x, y) \mathrm d x \mathrm d y.$$
However, this function is not continuous in the $\delta_{\square}$ metric on $\mathcal{W}^2$. In particular, on the set $W_1= W_2 = W$ the function $W \rightarrow \int_{[0, 1]^2} W(x, y)^2 \mathrm d x \mathrm d y$ is not continuous (see \cite[Example C.3.]{janson2010graphons}).

\subsection{Decomposition of Multiplexes and Multiplexons}

To circumvent the issue discussed above, we decompose a multiplex/multiplexon in two different ways. These decompositions allow us to account for dependencies between the different layers of a multiplex. Towards this, let $\P([r])$ be the collection of all subsets of $[r]$ and $\P_0([r]) = \P([r]) \setminus \{\emptyset\}$ the collection of all non-empty subsets of $[r]$.  


\begin{definition}
\label{def:decomp} 
For any $r$-multiplex $\bm{H} = (V(\bm H), H_1, H_2, \ldots, H_r)$, we define the \emph{disjoint decomposition} of $\bm{H}$ to be the $(2^r-1)$-multiplex $\hat{\bm{H}}$ given by
\begin{align}\label{eq:hatH}
\hat{\bm{H}} \defeq (V(\bm H), (\hat{H}_S)_{S \in \P_0([r])}), 
\end{align}
where for each $S \in \P_0([r])$, $\hat{H}_S = (V(\bm{H}), E(\hat{H}_S))$ and
\[E(\hat{H}_S) \defeq \{ \{i, j\}\subseteq V(\bm{H}): (i, j) \in E(H_s) \te{ for all } s \in S \text{ and } (i, j) \notin E(H_s) \te{ for all } s \notin S \} .\] 
Note that in a disjoint decomposition, for $S\neq S'$, $E(\hat{H}_S) \cap E(\hat{H}_{S'}) = \emptyset$. 
\end{definition}

In Figure \ref{fig:h}, we show the disjoint decomposition for the 2-layered multiplex 
\begin{align}\label{eq:H123}
\bm H = (\{1, 2, 3\}, E(H_1), E(H_2) ) , 
\end{align} 
where $H_1=  (\{1, 2, 3\}, \{(1, 2), (2, 3), (1, 3)\})$ is the triangle (colored in blue) and $H_2=  (\{1, 2, 3\}, \{(2, 3), $ $(1, 3)\})$ is the 2-star (colored in red). Note that 
\begin{itemize} 

\item the edge $(1, 2)$ is only in layer 1, hence $\hat H_{\{1\}}$ is a single edge and an isolated vertex, 

\item there are no edges that belong only in layer 2, hence $\hat H_{2}$ is an empty graph with 3 isolated vertices, 

\item the edges $(2, 3), (1, 3)$ are in both layer 1 and layer 2, hence $\hat H_{\{1, 2\}}$ is a 2-star. 

\end{itemize}

\begin{figure}
    \centering
    \includegraphics[scale=0.65]{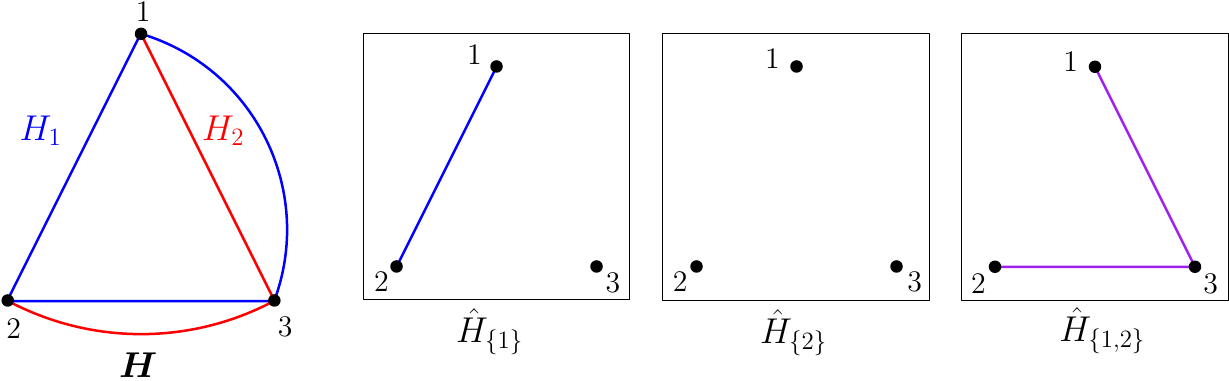}
    \caption{\small{ The disjoint decomposition of the 2-layered multiplex $\bm{H} = (H_1, H_2)$. }}
    \label{fig:h}
\end{figure}

Next, we present another natural way to decompose the edges of a multiplex network. 


\begin{definition}
\label{def:cumuldecomp} 
For any $r$-multiplex $\bm{H} = (V(\bm H), H_1, H_2, \ldots, H_r)$, we define the \emph{cumulative decomposition} of $\bm{H}$ to be the $(2^r-1)$-multiplex $\widebar{\bm{H}}$ given by
\[\widebar{\bm{H}} \defeq (V(\bm H), (\widebar{H}_S)_{S \in \P_0([r])}),\]
where for each $S \in \P_0([r])$, $\widebar{H}_S = (V(\bm{H}), E(\widebar{H}_S))$ and
\[E(\widebar{H}_S) \defeq \{ \{i, j\}\subseteq V(\bm{H}): (i, j) \in E(H_s) \te{ for all } s \in S  \} .\] 
Note that in a cumulative decomposition, for $S\subseteq S'$, $E(\widebar H_S) \subseteq E(\widebar H_{S'})$.
\end{definition}


In Figure \ref{fig:hc} we show the cumulative decomposition of the multiplex $\bm H$ in \eqref{eq:H123}. Note that 
\begin{itemize} 

\item the edges $(1, 2), (2, 3), (1, 3)$ are in layer 1, hence $\widebar{H}_{\{1\}}$ is a triangle, 

\item the edges $(2, 3), (1, 3)$ are in layer 2, hence $\widebar{H}_{\{2\}}$ is a 2-star, 

\item the edges $(2, 3), (1, 3)$ are in both layer 1 and layer 2, hence $\widebar{H}_{\{1, 2\}}$ is a 2-star. 

\end{itemize}

\begin{figure}
    \centering
    \includegraphics[scale=0.65]{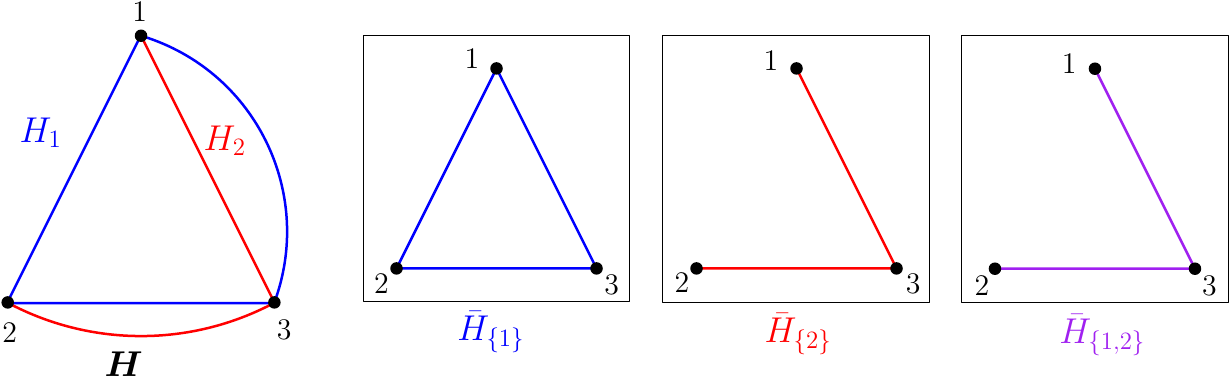}
    \caption{\small{ The cumulative decomposition of the 2-layered multiplex $\bm{H}=(H_1, H_2)$. }}
    \label{fig:hc}
\end{figure}

\begin{remark}[Multilinks]\label{remark:mlink}
Another alternative way to encode the layer structure of a multiplex network is through the notion of {\it multilinks} \cite[Section 7.3.2]{bianconi2018multilayer}. To explain this, consider an $r$-multiplex $\bm{H} = (V(\bm H), H_1, H_2, \ldots, H_r)$. The multilink between the nodes $i , j \in V(\bm H)$ is the $r$-dimensional binary vector 
$$\overrightarrow{m}_{ij} = (a_{ij}^{(1)}, a_{ij}^{(2)}, \ldots, a_{ij}^{(r)}),$$
where $a_{ij}^{(s)} = \bm 1 \{ (i, j) \in E(H_s)\}$, for $1 \leq s \leq r$. In other words, $\overrightarrow{m}_{ij}$ encodes the layers in $\bm{H}$ which contain the edge $(i, j)$. Let $$S_{\overrightarrow{m}_{ij}}  = \{s \in [r]: a_{ij}^{(s)}=1\}$$ be the subset of $[r]$ associated with the binary vector $\overrightarrow{m}_{ij}$. If we denote $\overrightarrow{H}_S = (V(\bm H), E(\overrightarrow{H}_S)) $ with $E(\overrightarrow{H}_S) = \{ 1 \leq i < j \leq |V(\bm G)|: S_{\overrightarrow{m}_{ij}}=S\}$, then $(V(\bm H), (\overrightarrow H_S)_{S \in \P_0([r])})$ is precisely the disjoint decomposition of $\bm{H}$ (recall \eqref{eq:hatH}). 
\end{remark}

Next, we define the decomposability of multiplexons. Towards this, denote by $\mathcal M(\P([r]))$ the collection of probability measure on $\P([r])$. 
Also, define $R := 2^r$. 

\begin{definition}\label{defintion:WR}
An $R$-multiplexon $\hat{\bm{W}} = (\hat W_S)_{S \in \P([r])} \in \wsp^{ R}$ is \emph{disjointly decomposable} if for almost every $(x,y) \in [0,1]^{2}$, 
\begin{align}\label{eq:hdisjoint}
\sum_{S \in \P([r])} \hat W_S(x,y) = 1. 
\end{align}
Denote by ${\wsp}_{\triangle}^{ R} \defeq \{\hat{\bm{W}} \in \wsp^{ R}:  \hat{\bm{W}} \te{ is disjointly decomposable}\}$. Furthermore, a $(R-1)$-multiplexon $\widebar{\bm{W}}  = (\widebar W_S)_{S \in \P_0([r])} \in \wsp^{ R-1}$ is \emph{cumulatively decomposable} if for almost every $(x,y) \in [0,1]^{2}$, there exists a probability measure $\mu_{x, y} \in \mathcal M(\P([r]))$ such that for all $S\in \P_0([r])$, 
\begin{align}\label{eq:Wxycumulative}
\sum_{S' : S'\supseteq S} \mu_{x, y} (S') = {\widebar W}_S(x, y). 
\end{align} 
Denote by ${\wsp}_{\blacktriangle}^{ R -1 } \defeq \{\widebar{\bm{W}} \in \wsp^{ R-1}:  \widebar{\bm{W}} \te{ is cumulatively decomposable}\}$. 
\end{definition}

One can construct a disjointly decomposable multiplexon from a cumulatively decomposable multiplexon and vice versa, as follows: 

\begin{itemize}

\item Given a cumulatively decomposable multiplexon $\widebar{\bm{W}} = (\bar W_S)_{S \in \P_0([r])} \in {\wsp}_{\blacktriangle}^{ R -1 }$, define its disjoint decomposition $\hat{\bm{W}} = (\hat W_S)_{S \in \P([r])} \in {\wsp}_{\triangle}^{ R }$ as follows: 
\begin{align}\label{eq:cd}
\hat{W}_{S} \defeq \sum_{S'\supseteq S} (-1)^{|S'\setminus S|} \widebar {W}_{S'}, \quad \text{ for } S\neq \emptyset , 
\end{align}
and $\hat{W}_{\emptyset} := 1 - \sum_{S \in \P_0([r])} \hat{W}_{S}$. Note that, by definition,  $\sum_{S \in \P([r])} \hat W_S = 1$. Hence, for \eqref{eq:cd} to be well-defined, we need verify that $\hat{W}_{S} \in [0, 1]$, for each $S\neq \emptyset$. To this end, from \eqref{eq:Wxycumulative}, we have, for $S\neq \emptyset$, 
\begin{align}\label{eq:Wcumulative}
\hat{W}_{S}(x, y) & = \sum_{S'\supseteq S} (-1)^{|S'\setminus S|} \sum_{S'' : S''\supseteq S'} \mu_{x, y} (S'') \nonumber \\  
& =  \sum_{S'' \in \P_0([r])} \mu_{x, y} (S'') \sum_{S' \in \P_0([r])} (-1)^{|S'|-|S|} \bm 1\{ S \subseteq S' \subseteq S''\} \nonumber \\ 
& = \sum_{S'' \in \P_0([r])} \mu_{x, y} (S'')  \bm 1\{ S \subseteq S''\}  \sum_{s=|S|}^{|S''|} (-1)^{s-|S|} {{|S''|-|S| \choose s-|S|}} \nonumber \\ 
& = \sum_{S'' \in \P_0([r])} \mu_{x, y} (S'') \bm 1\{ S \subseteq S''\} \sum_{t=0}^{|S''|-|S|} (-1)^{t} {{|S''|-|S| \choose t}} \nonumber \\ 
& = \sum_{S'' \in \P_0([r])} \mu_{x, y} (S'') \bm 1 \{S''= S\} = \mu_{x, y}(S)  \in [0, 1]. 
\end{align} 
This shows that under \eqref{eq:cd}, $\hat{\bm{W}} = (\hat W_S)_{S \in \P([r])} \in \wsp^{R-1}_{\triangle}$ is disjointly decomposable. 

\item Conversely, given a disjointly decomposable multiplexon $\hat{\bm{W}} = (\hat W_S)_{S \in \P([r])} \in {\wsp}_{\triangle}^{ R }$, define its cumulative decomposition $\widebar{\bm{W}} = (\widebar W_S)_{S \in \P_0([r])} \in {\wsp}_{\blacktriangle}^{ R -1 }$ as follows: 
\begin{align}\label{eq:dc}
\widebar{W}_{S} \defeq \sum_{S'\supseteq S} \hat {W}_{S} , 
\end{align}
for $S\neq \emptyset$.  For \eqref{eq:dc} to be well defined, we need to show that \eqref{eq:Wxycumulative} holds. To this end, define $\mu_{x, y}(S) = \hat{W}_{S}(x, y)$, for $S \in \P([r])$ and $x, y \in [0,1]$. Note that $\mu_{x, y} \in \mathcal M(\P_0([r]))$, since $\sum_{S \in \P([r])}\mu_{x, y}(S)  = \sum_{S \in \P([r])} \hat{W}_{S}(x, y)  = 1$. Hence, $\widebar{\bm{W}} = (\widebar W_S)_{S \in \P_0([r])}$ as in \eqref{eq:dc} is cumulatively decomposable. 
\end{itemize} 

\begin{remark}[Cut distance under different decompositions]
The representations \eqref{eq:cd} and \eqref{eq:dc} describe a bijection between $\wsp^{R-1}_{\blacktriangle}$ and $\wsp^{R-1}_{\triangle}$. In fact, this bijection is bicontinuous. Consider two cumulatively decomposable multiplexons $\widebar{\bm{W}}_1$ and $\widebar{\bm{W}}_2$ and their respective disjoint decompositions $\hat{\bm{W}}_1$ and $\hat{\bm{W}}_2$. Then the following holds: 
\begin{align}\label{eq:Wdistance}
\frac{1}{R}\|\hat{\bm{W}}_1 - \hat{\bm{W}}_2\|_{\square}\leq \|\widebar{\bm{W}}_1 - \widebar{\bm{W}}_2\|_{\square} \leq R\|\hat{\bm{W}}_1 - \hat{\bm{W}}_2\|_{\square}. 
\end{align}
The upper bound in \eqref{eq:Wdistance} follows from \eqref{eq:dc} and the triangle inequality: 
\begin{align*}
\|\widebar{\bm{U}} - \widebar{\bm{W}}\|_{\square} = \sum_{S \in \P_0([r])} \| \widebar{U}_S - \widebar{W}_S\|_{\square} & \leq \sum_{S \in \P_0([r])} \sum_{S' \supseteq S} \| \hat{U}_{S'} - \hat{W}_{S'} \|_{\square} \nonumber \\ 
& \leq R \sum_{S \in \P_0([r])} \| \hat{U}_{S} - \hat{W}_{S} \|_{\square} \nonumber \\ 
& \leq R \|\hat{\bm{U}} - \hat{\bm{W}}\|_{\square} . 
\end{align*}
For the lower bound \eqref{eq:Wdistance}, using \eqref{eq:cd} and the triangle inequality, it follows that
\begin{align*}
\sum_{S \in \P_0([r])} \| \hat{U}_S - \hat{W}_S\|_{\square} & \leq R \|\widebar{\bm{U}} - \widebar{\bm{W}}\|_{\square} . 
\end{align*}
Hence, 
\begin{align}\label{eq:UWcd}
\|\hat{\bm{U}} - \hat{\bm{W}}\|_{\square} = \sum_{S \in \P([r])} \| \hat{U}_S - \hat{W}_S\|_{\square} \leq 2 \sum_{S \in \P_0([r])} \| \hat{U}_S - \hat{W}_S\|_{\square} \leq 2 R \|\widebar{\bm{U}} - \widebar{\bm{W}}\|_{\square}, 
\end{align} 
where the first inequality in \eqref{eq:UWcd} follows by observing that  
$$\left\|  \hat U_{\emptyset} - \hat W_{\emptyset} \right\|_{\square} = \left\|  \sum_{S \in \P_0([r])} \left( \hat U_S - \hat W_S \right) \right\|_{\square} \leq \sum_{S \in \P_0([r])}  \left\| \hat U_S - \hat W_S \right\|_{\square} . $$
\end{remark}

Next, we define the empirical multiplexon associated with a multiplex. 

\begin{definition} (Empirical multiplexon) \label{defn:WGn}
Given an $r$-multiplex $\bm{G} = (V(\bm G), G_1, G_2, \ldots, G_r)$, the empirical multiplexon $\widebar{\bm{W}}^{\bm{G}} = ( \widebar{W}^{\bm{G}}_S)_{S \in \P_0([r])}$ is defined as 
\begin{align}\label{eq:WGSxy}
\widebar{W}^{\bm{G}}_S(x,y) = \bm 1\{ (\lceil x |V(\bm G)| \rceil,\lceil y |V(\bm G)|\rceil) \in E(\widebar{G}_S) \} , 
\end{align} 
where $\widebar{\bm{G}} = (\widebar{G}_S)_{S \in \P_0([r])}$ is the cumulative decomposition of $\bm{G}$. 
\end{definition}

\begin{figure}
    \centering
    \includegraphics[scale=0.65]{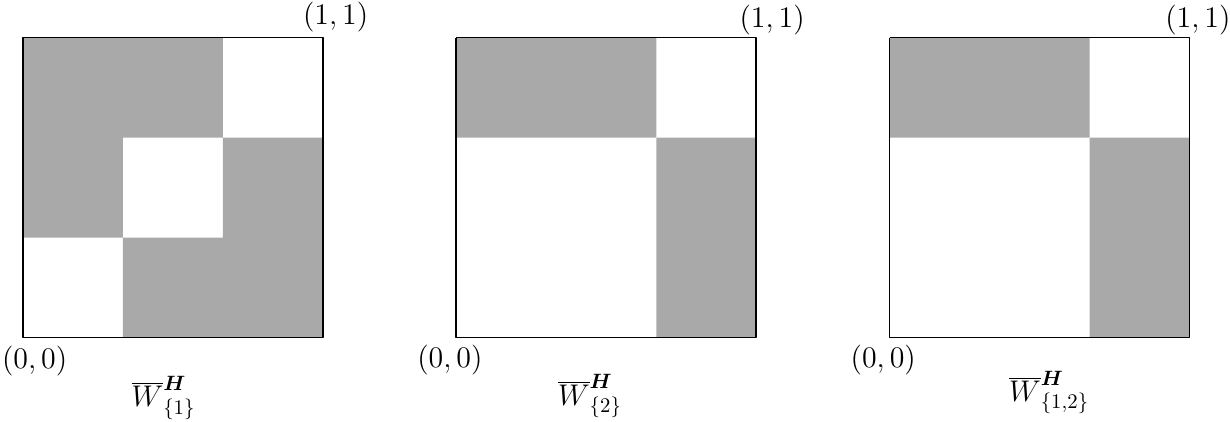}
    \caption{The empirical multiplexon corresponding to the multiplex $\bm H$ in \eqref{eq:H123}. }
    \label{fig:multiplexon}
\end{figure}

  Figure \ref{fig:multiplexon} shows the empirical multiplexon corresponding to the multiplex $\bm H$ in \eqref{eq:H123}. The graphons in the three layers, $\widebar W^{\bm{H}}_{\{ 1 \}}, \widebar W^{\bm{H}}_{\{ 2 \}}, \widebar W^{\bm{H}}_{\{ 1, 2 \}}$, are the empirical graphons corresponding to $\widebar{H}_{\{1\}}$, $\widebar{H}_{\{2\}}$, and $\widebar{H}_{\{1, 2\}}$, respectively, as in shown Figure \ref{fig:hc}. The grey regions are where the function values are 1.

\begin{remark}[Empirical multiplexons are cumulatively decomposable] Given an $r$-multiplex $\bm{G}$ as above, denote, for $x, y \in [0,1]$, 
$$L_{x, y} := \{s \in [r]: (\lceil x |V(\bm G)| \rceil,\lceil y |V(\bm G)|\rceil) \in E(G_s) \},$$
the set of layers which contain the edge $(\lceil x |V(\bm G)| \rceil,\lceil y |V(\bm G)|\rceil)$. Consider the measure $\mu_{x, y}$ which puts point mass on the set $L_{x, y}$, that is, 
$$\mu_{x, y} (S) = \bm 1\{S = L_{x, y}\},$$ 
for $S \in \P([r])$. Then, recalling \eqref{eq:WGSxy}, 
$$ \widebar{W}^{\bm{G}}_S(x,y) = \bm 1 \{ S \subseteq L_{x, y}\} = \sum_{S' : S'\supseteq S} \mu_{x, y}(S').$$ 
This shows that $\widebar{\bm{W}}^{\bm{G}}$ is cumulatively decomposable, that is, $\widebar{\bm{W}}^{\bm{G}} \in {\wsp}_{\blacktriangle}^{R-1}$. 
\end{remark}

We now relate the homomorphism of density between 2 multiplexes $\bm H$ and $\bm G$, in terms of the homomorphism density of the disjoint decomposition $\hat{\bm{H}}$ of $\bm H$ in the empirical multiplexon $\widebar{\bm{W}}^{\bm G}$. 

\begin{proposition}
\label{proposition:gtow}
For any two $r$-multiplexes $\bm{H} = (V(\bm H), H_1, H_2, \ldots, H_r)$ and $\bm{G} = (V(\bm G),$ $ G_1,G_2, \ldots, G_r)$,
\begin{align}\label{eq:gtow}
t(\bm{H},\bm{G}) = t(\hat{\bm{H}},\widebar{\bm{G}}) = t(\hat{\bm{H}}, \widebar{\bm{W}}^{\bm G}) , 
\end{align}
where $\hat{\bm{H}}$ is the disjoint decomposition of $\bm{H}$, $\widebar{\bm{G}}$ is the cumulative decomposition of $\bm{G}$, and $\widebar{\bm{W}}^{\bm G}$ is the empirical multiplexon of $\bm{G}$. 
\end{proposition} 
 
\begin{proof} We start by proving the first equality in \eqref{eq:gtow}. Noting that $V(\bm{H}) = V(\hat{\bm{H}})$ and $V(\bm{G}) = V(\widebar{\bm{G}})$, it suffices to show that $\phi: V(\bm{H}) \to V(\bm{G})$ is a homomorphism from $\bm{H}$ to $\bm{G}$ if and only if it is a homomorphism from $\hat{\bm{H}}$ to $\widebar{\bm{G}}$. 
\begin{itemize}

\item For the forward direction, suppose $\phi$ is a homomorphism from $\bm{H}$ to $\bm{G}$. This means, recalling Definition \ref{definition:homomorphismGH}, $\phi$ is a homomorphism from the graph $H_s = (V(\bm H), E(H_s))$ to the graph $G_s = (V(\bm G), E(G_s))$, for each $1 \leq s \leq r$. Now, fix a non-empty subset $S \in \P_0([r])$ and suppose $e = (u, v) \in E(\hat{H}_S)$, for $u, v \in V(\bm{H})$. Then for each $s \in S$, we have $e \in E(H_s)$, hence, $(\phi(u), \phi(v)) \in E(G_s)$. Since this holds for all $s \in S$, $(\phi(u), \phi(v)) \in E(\widebar G_S)$, which implies that $\phi$ is a homomorphism from $\hat{H}_S$ to $\widebar{G}_S$. 

\item In the reverse direction, suppose $\phi$ is a homomorphism from $\hat{\bm{H}}$ to $\widebar{\bm{G}}$. Then for any $s \in [r]$ and $e = (u, v) \in E(H_s)$, from the definition of disjoint decomposition, there exists a unique subset $T_s \in \P_0([r])$ containing $s$, such that $e \in E(\hat H_{T_s})$. Hence, by Definition \ref{definition:homomorphismGH}. $\phi$ must be a homomorphism from $\hat H_{T_s}$ to $\widebar G_{T_s}$, which means, $(\phi(u), \phi(v)) \in E(\widebar G_{T_s})$. Since $s \in T_s$, this implies $(\phi(u), \phi(v)) \in E(G_s)$. This shows, $\phi$ is a homomorphism from $H_s$ to $G_s$, for any $s \in [r]$. Hence, $\phi$ is a homomorphism from $\bm{H}$ to $\bm{G}$. 
\end{itemize}

Next, we proceed to prove the second equality in \eqref{eq:gtow}. Toward this, for $i \in [|V(\bm{G})|]$ denote 
\begin{align}\label{eq:interval}
Q_i= \left( \frac{i-1}{|V(\bm{G})|}, \frac{i}{|V(\bm{G})|} \right] , 
\end{align}
 and, for $K \geq 1$ and $i_1,\dots,i_K \in [|V(\bm{G})|]$, define
\begin{equation*}
Q_{i_1\dots i_K} \defeq \bigtimes_{L=1}^{K} Q_{i_{L}}\subseteq [0,1]^{K}.
\end{equation*} 
Then, with $K= |V(\bm{H})|$ and recalling \eqref{eq:tHW},  
\begin{align*}
& t(\hat{\bm{H}},\bm{W}^{G}) \\ 
&= \int_{[0,1]^{|V(\bm{G})|}} \prod_{S \in \P_0([r])} \prod_{(i, j) \in E(\hat{H}_S)} \widebar{W}_S^{\bm G}(x_i, x_j)\, \prod_{i=1}^K \mathrm d x_i\\
&= \sum_{(i_1,\dots,i_{K}) \in [|V(\bm{G})|]^{K}}\int_{Q_{i_1\dots i_{K}}} \prod_{S \in \P_0([r])} \prod_{(i, j) \in E(\hat{H}_S)} \widebar{W}_S^{\bm G}(x_i,x_j)\, \prod_{i=1}^K \mathrm d x_i\\
&= \sum_{\phi: V(\bm{H}) \to V(\bm{G})}\int_{Q_{\phi(1)\dots \phi(K)}} \prod_{S \in \P_0([r])} \prod_{(i, j) \in E(\hat{H}_S)} \widebar{W}_S^{\bm G}(x_i,x_j)\, \prod_{i=1}^K \mathrm d x_i\\
&= \sum_{\phi: V(\bm{H}) \to V(\bm{G})}\int_{Q_{\phi(1)\dots \phi(K)}} \prod_{S \in \P_0([r])} \prod_{(i, j) \in E(\hat{H}_S)} \bm 1\{(x_i,x_j) \in Q_{\phi(i)\phi(j)} \} \widebar{W}_S^{\bm G}(x_i,x_j)\, \prod_{i=1}^K \mathrm d x_i\\
&= \sum_{\phi: V(\bm{H}) \to V(\bm{G})}\int_{Q_{\phi(1)\dots \phi(K)}} \prod_{S \in \P_0([r])} \prod_{(i, j) \in E(\hat{H}_S)} \bm 1\{( \phi(i),\phi(j)) \in E(\widebar{G}_S)\} \, \prod_{i=1}^K \mathrm d x_i\\
&=   \sum_{\phi: V(\bm{H}) \to V(\bm{G})} \frac{\bm 1\{\phi \in \mathrm{hom}(\hat{\bm{H}},\widebar{\bm{G}})\}}{|V(\bm{G})|^{K}} = t(\hat{\bm{H}},\widebar{\bm{G}}).
\end{align*} 
This completes the proof of Proposition \ref{proposition:gtow}. 
\end{proof}

As an illustration of the above result, suppose $\bm{H}$ is the 2-layered multiplex in \eqref{eq:H123} and $\bm{G} = ([n], G_1, G_2 )$ is a 2-layered multiplex with vertex set $[n]$. If the adjacency matrices of the 2 layers $G_1$ and $G_2$ are denoted by $A_{G_1} = (a_{ij}^{(1)})_{1 \leq i, j \leq n}$ and $A_{G_2} = (a^{(2)}_{ij})_{1 \leq i, j \leq n}$, then one can express the homomorphism density as: 
\begin{align}\label{eq:H123W}
t(\bm H, \bm G) = \frac{1}{n^3} \sum_{i_1, i_2, i_3 \in [n] } a^{(1)}_{i_1 i_2} a^{(1)}_{i_2 i_3} a^{(2)}_{i_2 i_3} a^{(1)}_{i_1 i_3} a^{(2)}_{i_1 i_3}. 
\end{align}
Note that $a^{(1)}_{i_1 i_2} = \bm 1\{(i_1, i_2) \in E(\widebar G_{\{1\}}) \}$, $a^{(1)}_{i_2 i_3} a^{(2)}_{i_2 i_3} =  \bm 1\{(i_1, i_2) \in E(\widebar G_{\{1, 2\})}\}$, and $a^{(1)}_{i_2 i_3} a^{(2)}_{i_2 i_3} =  \bm 1\{(i_1, i_3) \in E(\widebar G_{\{1, 2\}})\}$. Hence, recalling the definition of the empirical multiplexon from \eqref{eq:WGSxy}, the RHS of \eqref{eq:H123W} can expressed as: 
$$t(\bm H, \bm G)= \int_{[0, 1]^3} \widebar{W}_{\{1\}}^{\bm G}(x_1, x_2) \widebar{W}_{\{1, 2\}}^{\bm G}(x_2, x_3) \widebar{W}_{\{1, 2\}}^{\bm G}(x_1, x_3) \mathrm d x_1 \mathrm d x_2 \mathrm d x_3= t(\hat{\bm{H}},\widebar{\bm{W}}^{\bm G}),$$  
where the second equality above follows from \eqref{eq:tHW}. This verifies \eqref{eq:gtow} when $\bm{H}$ is the 2-layered multiplex in \eqref{eq:H123}.
 
\section{The Space of Cumulatively Decomposable Multiplexons}
\label{sec:compactness} 


The cumulative decomposition of a multiplexon provides a convenient way to represent homomorphism densities (as shown in Proposition \ref{proposition:gtow}). Consequently, the space ${\wsp}_{\blacktriangle}^{R-1}$ of cumulatively decomposable multiplexons is a natural candidate for representing the limit objects of convergent sequences of multiplex networks. In this section, we collect various important properties of the space ${\wsp}_{\blacktriangle}^{R-1}$. We begin by showing that ${\wsp}_{\blacktriangle}^{R-1}$ is closed under the cut norm.

\begin{proposition}
\label{ppn:closedcumdec}
For any $r \geq 1$ and $R=2^{r}$, the space ${\wsp}_{\blacktriangle}^{R-1}$ is closed under the cut norm $\|\cdot\|_\square$ (as defined in \eqref{eq:UWsquarelayer}). 
\end{proposition}

\begin{proof} If $r = 1$ (that is, $R-1=1$), then all $1$-multiplexons are cumulatively decomposable, so the result holds by completeness of $\wsp$. 

Fix $r \geq 2$ and let $\{\widebar{\bm{W}}_n\}_{n \geq 1} = \{ ((\widebar{W}_n)_S)_{S \in \P_0([r])}\}_{n \geq 1}$ be a sequence of cumulatively decomposable multiplexons converging in the cut norm to a $(R-1)$-multiplexon $\widebar{\bm{W}} = ((\widebar{W})_S)_{S \in \P_0([r])}$. We will show that $\widebar{\bm{W}} \in {\wsp}_{\blacktriangle}^{R-1}$. To this end, recalling \eqref{eq:UWsquarelayer}, note that 
$$\|(\widebar{W}_n)_S - (\widebar{W})_S\|_\square \leq \|\widebar{\bm{W}}_n - \widebar{\bm{W}}\|_\square \rightarrow 0 $$
for any $S \in \P_0([r])$. Now, let $\hat{\bm{W}}_n = ((\hat{W}_n)_S)_{S \in \P_0([r])}$ be the disjoint decomposition of $\widebar{\bm{W}}_n$ (as in \eqref{eq:cd}). 
Then, by \eqref{eq:cd} and the triangle inequality, 
$$\|(\hat{W}_n)_S) - (\hat{W})_S\|_\square \rightarrow 0 , $$
where $\hat{W}_{S} = \sum_{S'\supseteq S} (-1)^{|S'\setminus S|} \widebar {W}_{S'}$, for $S \ne \emptyset$. Furthermore, since $\hat{(W_n)}_S \geq 0$, for $S \in \P_0([r])$, and $\sum_{S \in \P_0([r])} \hat{(W_n)}_S \leq 1$, we have $\hat{(W)}_S \geq 0$, for $S \in \P_0([r])$, and $\sum_{S \in \P_0([r])} \hat{(W)}_S \leq 1$. Hence, defining $\hat W_{\emptyset} = 1 -\sum_{S \in \P_0([r])} \hat{(W)}_S$, we get $\hat{\bm W} = ((\hat{W})_S)_{S \in \P([r])} \in \wsp^{R}_{\triangle}$. This implies, from \eqref{eq:dc}, $\widebar{\bm{W}}\in {\wsp}_{\blacktriangle}^{R-1}$. This completes the proof of Proposition \ref{ppn:closedcumdec}. 
\end{proof}

A central result in graph limit theory is the compactness of the quotient space of graphons under the cut distance  \cite{lovasz2007szemeredi} (see also \cite[Section 9.3]{Lov12}). In the following theorem, we establish compactness for the quotient space of cumulatively decomposable multiplexons. To this end, for $r \geq 1$ and $R = 2^{r}-1$, denote by $({\wsp}_{\blacktriangle}^{R-1})_{\square}$ the collection of equivalence classes $\{\tilde{\widebar{{\bm W}}}: \widebar{\bm W} \in {\wsp}_{\blacktriangle}^{R-1}\}$ associated with the metric $\delta_\square$ (as defined in \eqref{eq:UWdeltalayer}).  
Here, $\tilde{\widebar{{\bm W}}}$ denotes the orbit of $\widebar{\bm W} \in {\wsp}_{\blacktriangle}^{R-1}$ under the equivalence relation $\sim$ (defined after \eqref{eq:UWdeltalayer}).

\begin{theorem}
\label{cumdeccom}
The space $({\wsp}_{\blacktriangle}^{R-1})_{\square}$ is well-defined and compact under the metric $\delta_\square$.  
\end{theorem}

\begin{proof}[Proof of Theorem \ref{cumdeccom}] Fix an integer $L \geq 1$. Recall from Section \ref{sec:distancelayers} that $({\wsp}^L)_{\square}$ is the set of $L$-multiplexons where any pair of $L$-multiplexons with cut distance zero are considered the same point in the space. The proof of Theorem \ref{cumdeccom} relies on the following general result (see Appendix \ref{sec:compactpf} for a proof):

\begin{proposition}
\label{lem:compact}
For any integer $L \geq 1$, the space $({\wsp}^{L})_\square$ is compact.
\end{proposition}

Using the above result, we can easily complete the proof of Theorem \ref{cumdeccom}. We first show that $({\wsp}_{\blacktriangle}^{R-1})_{\square}$ is a well-defined space. For that we must prove that for any $\widebar{\bm{W}}_1\sim \widebar{\bm{W}}_2$, 
\begin{align}\label{eq:W12}
\widebar{\bm{W}}_1\in {\wsp}_{\blacktriangle}^{R-1} \text{ if and only if } \widebar{\bm{W}}_2\in {\wsp}_{\blacktriangle}^{R-1} . 
\end{align} 
To this end, for $\bm{W} = (W_S)_{\P_0([r])} \in \wsp^{R-1}$, define
\[\Emc(\bm{W}) \defeq \left\{(x,y)\in [0,1]^{2}: \exists~\mu_{xy} \in \mathcal M (\P([r])) \text{ such that } W_S(x,y) = \sum_{S'\supseteq S} \mu_{xy}(S'), \te{ for } S \in \P_0([r]) \right\}.\] 
Note that a multiplexon $\bm{W}$ is an element of ${\wsp}_{\blacktriangle}^{R-1}$ if and only if $\Emc(\bm{W})$ is a set of full measure. Now, we proceed to prove \eqref{eq:W12}. To begin with, suppose $\widebar{\bm{W}}_1 = ((\widebar{W}_1)_S)_{S \in \P_0([r])} \in{\wsp}_{\blacktriangle}^{R-1}$. Then for any measure preserving bijection $\sigma \in \mpt$, define 
$$\sigma^{-1}(\Emc(\widebar{\bm{W}}_1)) = \{( \sigma^{-1}(x), \sigma^{-1}(y)): (x, y) \in  \Emc(\widebar{\bm{W}}_1) \}.$$
Hence, for each $S \in \P_0([r])$, and $(x,y) \in \sigma^{-1}(\Emc(\widebar{\bm{W}}_1))$ (which means $(\sigma(x), \sigma(y)) \in \Emc(\widebar{\bm{W}}_1)$),
$$(\widebar{W}_1)_S^\sigma(x,y) = (\widebar{W}_1)_S(\sigma(x),\sigma(y)) = \sum_{ S'\supseteq S} \mu_{\sigma(x)\sigma(y)}(S'), $$ 
since $\Emc(\widebar{\bm{W}}_1)$ has full measure. Thus, $\Emc(\widebar{\bm{W}}^{\sigma}_1) \supseteq \sigma^{-1}(\Emc(\widebar{\bm{W}}_1))$ is a set of full measure, which means $\widebar{\bm{W}}^{\sigma}_1 \in {\wsp}_{\blacktriangle}^{R-1}$. Now, fix any $\widebar{\bm{W}}_2 \sim \widebar{\bm{W}}_1$. Then there exists a sequence $\{\sigma_n\}_{n \geq 1} \subset\mpt$ such that $\lim_{n\to\infty} \| \widebar{\bm{W}}_1^{\sigma_n} - \widebar{\bm{W}}_2 \|_\square =0$. Since each $\widebar{\bm{W}}_1^{\sigma_n}\in {\wsp}_{\blacktriangle}^{R-1}$  and the space ${\wsp}_{\blacktriangle}^{R-1}$ is closed (by Proposition \ref{ppn:closedcumdec}), it follows that $\widebar{\bm{W}}_2\in {\wsp}_{\blacktriangle}^{R-1}$. The reverse implication holds by an identical argument. This concludes the proof that $({\wsp}_{\blacktriangle}^{R-1})_{\square}$ is well-defined. 

To show compactness, suppose $\{\tilde{\widebar{\bm{W}}}_n\}_{ n \geq 1 }$ is a sequence of elements in $({\wsp}_{\blacktriangle}^{R-1})_{\square}$. Since $({\wsp}_{\blacktriangle}^{R-1})_{\square}\subseteq (\wsp^{R-1})_{\square}$, by Lemma \ref{lem:compact} there exists a subsequence $\{n_m\}_{m \geq 1}$ and $\tilde{\widebar{\bm{W}}}\in (\wsp^{R-1})_{\square}$ such that $\tilde{\widebar{\bm{W}}}_{n_m} \to \tilde{\widebar{\bm{W}}}$ in $(\wsp^{R-1})_{\square}$. This implies that there exists a sequence $\{\widebar{\bm{W}}_{n_m}\}_{m \geq 1} \subset {\wsp}_{\blacktriangle}^{R-1}$ with $\widebar{\bm{W}}_{n_m}\sim \tilde{\widebar{\bm{W}}}_{n_m}$ for every $m \geq 1$ and a multiplexon $\widebar{\bm{W}}\sim \tilde{\widebar{\bm{W}}}$ such that $\lim_{m\to\infty} \|\widebar{\bm{W}}-\widebar{\bm{W}}_{n_m}\|_{\square} = 0$. However, since ${\wsp}_{\blacktriangle}^{R-1}$ is closed, $\widebar{\bm{W}} \in {\wsp}_{\blacktriangle}^{R-1}$. Thus, $\tilde{\widebar{\bm{W}}} \in ({\wsp}_{\blacktriangle}^{R-1})_{\square}$. This completes the proof of the compactness of $({\wsp}_{\blacktriangle}^{R-1})_{\square}$. 
\end{proof}

\begin{remark}\label{remark:compactpf}
Versions of Proposition \ref{lem:compact} and its generalizations appear in the literature in different contexts. In particular, this is established for finitely decorated graphs in \cite[Theorem 5.5]{falgas2016multicolour} and for multirelational graphs in \cite[Theorem 3]{alvarado2023limits}. Proposition \ref{lem:compact} can also be realized as a special case of the compactness result for the space of probability graphons \cite[Theorem 1.3]{abraham2025probability} (see Section \ref{sec:probabilitygraphons}). 
In Appendix \ref{sec:compactpf} we give a self-contained proof of Proposition \ref{lem:compact} that extends the arguments from the graphon setting to multiplexons. The main technical step is to ensure that there is a common partition of $[0, 1]$ for which all the layers of a multiplexon can be simultaneously approximated in the cut distance by a block multiplexon. This is achieved by considering the refinement of the partitions obtained by applying Szemer\'edi's regularity lemma on each of the layers. 
\end{remark} 

\section{ Sampling Random Multiplexes }
\label{sec:randommultiplex}

Analogous to the role of graphons in generating inhomogeneous random graphs, multiplexons provide a natural framework for sampling inhomogeneous random multiplex networks.

\begin{definition}[$\bm{W}$-random multiplexes] 
\label{defn:Wrandomlayers} 
Fix $r \geq 1$ and consider a disjointly decomposable multiplexon $\hat{\bm{W}} = ( \hat{W}_S)_{S \in \P([r])}$. Then one can generate a random $r$-multiplex $\bm{G}(n, \hat{\bm{W}})$, with vertex set $V(\bm{G}_n) = [n]$ and disjoint decomposition $([n], ((\hat{G}_n)_S)_{S \in \P_0([r])})$, as follows: 
\begin{itemize} 
\item Generate $\eta_1, \eta_2, \ldots, \eta_n$ i.i.d. $\mathrm{Unif}([0, 1])$. 
\item Independently for every pair $1 \leq i < j \leq n$, there exists a $\P([r])$-random element $S_{(i,j)}$ whose distribution is given by the probability measure $\sum_{S \in \P([r])} \hat{W}_S(\eta_i,\eta_j)\delta_S$. Then, $(i,j) \in E(\hat{H}_{S_{(i,j)}})$. This implies that
\begin{align}\label{eq:Wrandom}
\mathbb P((i, j) \in E((\hat{G}_n)_S)) = \hat{W}_S(\eta_i, \eta_j), \text{ for } S \in \P_0([r]), 
\end{align}
and $\hat{W}_{\emptyset} (\eta_i,\eta_j)$ is probability that the edge $(i, j)$ is not present in any layer. 
\end{itemize} 
With slight abuse of notation, we will refer to  $\bm{G}(n, \hat{\bm{W}})$ as a $\bm{W}$-random multiplex. 
\end{definition}

Note that any graphon $W$ has a disjoint decomposition with $r=1$ as follows: $\hat{\bm{W}} = (\hat W_{\emptyset}, \hat W_{\{1\}})$, where $\hat W_{\{1\}} := W$ and $\hat W_{\emptyset} := 1 - W$. Hence, when $r=1$, Definition \ref{defn:Wrandomlayers} becomes the well-known $W$-random graph model $G(n, W)$, where each edge $(i, j)$ is present with probability $W(\eta_i, \eta_j)$, independently for $1 \leq i < j \leq n$. 

\begin{remark}\label{remark:Wrandomlayerscumulative}
One can also consider the cumulative decomposition of $\bm{G}(n, \hat{\bm{W}})$, which takes the form: $([n], ((\widebar{G}_n)_S)_{S \in \P_0([r])})$, 
where, given $\eta_1, \eta_2, \ldots, \eta_n$ i.i.d. $\mathrm{Unif}([0, 1])$, 
\begin{align}\label{eq:Wcumulativerandom}
\mathbb P((i, j) \in E((\widebar{G}_n)_S)) =  \sum_{S' \supseteq S} \hat{W}_{S'}(\eta_i, \eta_j) :=  \widebar{W}_S(\eta_i, \eta_j) , \text{ for } S \in \P_0([r]) , 
\end{align}
independently for every pair $1 \leq i < j \leq n$. From \eqref{eq:dc} we know that $\widebar{\bm{W}} = ( \widebar{W}_S)_{S \in \P_0([r])}$ is a cumulatively decomposable multiplexon. 
Hence, given any cumulatively decomposable multiplexon $\widebar{\bm{W}} = ( \widebar{W}_S)_{S \in \P_0([r])}$, with slight abuse of notation we denote by $\bm{G}(n, \widebar{\bm{W}})$ the $r$-multiplex generated according to \eqref{eq:Wcumulativerandom}. Moreover, we will refer to both $\bm{G}(n, \widebar{\bm{W}})$ and $\bm{G}(n, \hat{\bm{W}})$ as $\bm{W}$-random multiplexes. 
Notable examples of the $\bm{W}$-random multiplex are the homogeneous Erd\H{o}s-R\'enyi random multiplex (see Remark \ref{remark:Gnpparameters}) and the correlated stochastic block model (Example \ref{example:blockgraphon}). Variations of this model have also appeared recently in graph-matching literature \cite{racz2023matching,song2023independence,matchingdegree}. 
\end{remark}

An important special case of Definition \ref{defn:Wrandomlayers} is when all the layers of the multiplexon are constant functions. This generates the multiplex analogue of the classical Erd\H{o}s-R\'enyi random graph (recall Example \ref{example:random12}).

\begin{remark}[Erd\H{o}s-R\'enyi random multiplex]
\label{remark:Gnpparameters}
Consider a disjointly decomposable multiplexon $\hat{\bm{W}} = ( \hat{W}_S)_{S \in \P([r])}$, 
such that 
$$\hat{W}_S(x, y) = \hat{p}_S \in [0, 1], \quad \text{ for } x, y \in [0, 1],$$ 
where $(\hat{p}_S)_{S \in \P([r])}$ are non-negative constants satisfying $\sum_{S \in \P([r])} \hat{p}_S = 1$. Then the random $r$-multiplex $\bm{G}(n, \hat{\bm{W}})$, as generated in Definition \ref{defn:Wrandomlayers}, has disjoint decomposition $([n], ((\hat{G}_n)_S)_{S \in \P_0([r])})$ as follows: 
$$\mathbb P((i, j) \in E((\hat{G}_n)_S)) = \hat{p}_S, \text{ for } S \in \P_0([r]) , $$ 
independently for every pair $1 \leq i < j \leq n$. In other words, for each $S \in \P_0([r])$, the simple graph $(\hat{G}_n)_S$ is marginally distributed as an Erd\H{o}s-R\'enyi random graph $G(n,  \hat{p}_S)$. This random multiplex will be referred to as the {\it Erd\H{o}s-R\'enyi $r$-multiplex} and will be denoted as $G(n, (\hat{p}_S)_{S \in \P_0([r])})$. 
Alternatively, one can consider the cumulative decomposition of 
$G(n, \hat{\bm{W}})$ as  $([n], ((\widebar{G}_n)_{\{ S \}})_{S \in \P_0([r])})$, where, independently for $1 \leq i < j \leq n$,  
$$\mathbb P((i, j) \in E(\widebar{G}_n)_{\{S\}})) = \sum_{S' \supseteq S} \hat{p}_{S'} := \widebar{p}_{S}  ,  \text{ for } S \in \P_0([r]),$$
where the first equality follows from \eqref{eq:Wcumulativerandom}. Hence, an alternative parametrization of the Erd\H{o}s-R\'enyi $r$-multiplex is $G(n, (\widebar{p}_S)_{S \in \P_0([r])})$. For $r=2$, the cumulative decomposition of $G(n, \hat{\bm{W}})$ simplifies to 
$$([n], ((\widebar{G}_n)_{\{ 1 \}}, (\widebar{G}_n)_{\{2\}}, (\widebar{G}_n)_{\{ 1, 2 \}}),$$ where, independently for $1 \leq i < j \leq n$,  
\begin{align*}
\mathbb P((i, j) \in E(\widebar{G}_n)_{\{1\}})) & = \mathbb P((i, j) \text{ is an edge in layer 1}) = \hat{p}_{\{1\}} + \hat{p}_{\{1, 2\}} = \widebar{p}_{\{1\}}, \nonumber \\
\mathbb P((i, j) \in E(\widebar{G}_n)_{\{2\}})) & = \mathbb P((i, j) \text{ is an edge in layer 2}) = \hat{p}_{\{2\}} + \hat{p}_{\{1, 2\}} = \widebar{p}_{\{2\}} , \nonumber 
\end{align*} 
and 
$$\mathbb P((i, j) \in E(\widebar{G}_n)_{\{1, 2\}}))= \mathbb P((i, j) \text{ is an edge both layers 1 and 2}) = \hat{p}_{\{1, 2\}} = \widebar{p}_{\{1, 2\}}.$$
Hence, the relation with the parametrization in Example \ref{example:random12} is as follows: $p_1= \hat{p}_{\{1\}} + \hat{p}_{\{1, 2\}} = \widebar{p}_{\{1\}}$, $p_2= \hat{p}_{\{2\}} + \hat{p}_{\{1, 2\}}= \widebar{p}_{\{2\}}$, and $p_{12}= \hat{p}_{\{1, 2\}} = \widebar{p}_{\{1,2\}}$. 
\label{example:randomlayers}
\end{remark}

Next, we show that homomorphism densities of a cumulatively decomposable multiplexon $\widebar{\bm{W}}\in{\wsp}_{\blacktriangle}^{R-1}$ are given by containment probabilities in the random multiplex sampled from $\widebar{\bm{W}}$.

\begin{lemma}
\label{ppn:char}
Suppose $\bm{H}$ is an $r$-multiplex and $\widebar{\bm{W}}\in{\wsp}_{\blacktriangle}^{R-1}$ is a cumulatively decomposable multiplexon. Then, for any $n \geq |V(\bm H)|$,  
\[\PP(\bm{H} \subseteq \bm{G}(n, \widebar{\bm{W}})) = t(\hat{\bm{H}}, \widebar{\bm{W}}) , \] 
where $\bm{G}(n, \widebar{\bm{W}})$ is as defined in Remark \ref{remark:Wrandomlayerscumulative}. 
\end{lemma}
\begin{proof}
Suppose $\{\eta_1,\dots,\eta_n\}$ are i.i.d. $\mathrm{Unif}([0, 1])$ random variables. Denote by $\mathcal F_n$ the sigma algebra generated by the collection $\{\eta_1,\dots,\eta_n\}$. Then, from \eqref{eq:Wcumulativerandom}, 
\begin{align*}
\PP\left(\bm{H}\subseteq  \bm{G}(n, \widebar{\bm{W}}) \middle| \mathcal F_n \right) &= \prod_{S \in \P_0([r])}\prod_{e \in E(\hat{H}_S)} \PP \left(e \in E(\widebar{G}_S)\middle| \mathcal F_n \right) 
= \prod_{S \in \P_0([r])}\prod_{(i, j) \in E(\hat{H}_S)} \widebar{W}_S (\eta_i,\eta_j).
\end{align*}
Thus, 
\begin{align*}
\PP\left(\bm{H}\subseteq \bm{G}(n, \widebar{\bm{W}}) \right) 
= \int_{[0,1]^{n}} \prod_{S \in \P_0([r])}\prod_{(i, j) \in E(\hat{H}_S)} \widebar{W}_S (x_i,x_j)\,\mathrm dx_1\,\dots\,\mathrm dx_n &= t(\hat{\bm{H}}, \widebar{\bm{W}}). 
\end{align*} 
This completes the proof. 
\end{proof}

Given a cumulative decomposable multiplexon $\widebar{\bm{W}} = (\widebar{W}_S)_{S \in \P_0([r])}$, denote by $\bm{G}(n, \widebar{\bm{W}})$ the random $r$-multiplex on $n$ vertices generated as in Remark \ref{remark:Wrandomlayerscumulative}. Then for any $r$-multiplexes $\bm{H} = (V(\bm H), H_1, H_2, \ldots, H_r)$, the injective homomorphsim density (recall  \eqref{eq:injectivetHGlayers} from Proposition \ref{proposition:gtow}) satisfies: 
\begin{align}\label{eq:thGrandomlayers}
\E[t_{\text{inj}}(\bm{H}, \bm{G}(n, \widebar{\bm{W}})] = t(\hat{\bm{H}}, \widebar{\bm{W}}) , 
\end{align}
and, from \eqref{eq:tHGnoninjective}, $\E[t(\bm{H}, \bm{G}(n, \widebar{\bm{W}})] = t(\hat{\bm{H}}, \widebar{\bm{W}}) + O(\frac{1}{n})$. Hence, by an application of the bounded difference inequality similar to the case of graphons (see \cite[Theorem 4.4.5]{zhao2023graph}), we can get 
\begin{equation}
\label{eq::tsamplebound}
\mathbb{P}( |t(\bm{H}, \bm{G}(n, \widebar{\bm{W}}) - t(\hat{\bm{H}}, \widebar{\bm{W}}) | \geq \varepsilon) \leq 2 e^{ -\frac{\varepsilon^2 n }{ 8 |V(\bm{H})| }} .
\end{equation}
This implies that the sequence $\{\bm{G}(n, \widebar{\bm{W}})\}_{ n \geq 1 }$ is left-convergent to the multiplexon $\widebar{\bm{W}}$ with probability 1. Further, the empirical graphon associated with $\bm{G}(n, \widebar{\bm{W}})$ converges in the cut distance to $\widebar{\bm{W}}$ with high probability. This is formalized in the following result. A more general version of this result for probability graphons is proved in \cite[Lemma 6.12]{abraham2025probability}.

\begin{proposition}
\label{ppn:complete}
Given $\widebar{\bm{W}}\in{\wsp}_{\blacktriangle}^{R-1}$, denote by $\widebar{\bm{W}}^{\bm{G_n}}$ the empirical graphon associated with the random $r$-multiplex $\bm{G_n} := \bm{G}(n, \widebar{\bm{W}})$. Then, for $n$ large enough, 
\begin{align}\label{eq:WGnW}
\delta_\square(\widebar{\bm{W}}^{\bm{G_n}}, \widebar{\bm W}) \leq \frac{32 R^{\frac{3}{2}}}{\sqrt{\log n}} ,  
\end{align} 
with probability at least  $1-e^{- \frac{ 2 R n}{\log n}}$. 
\end{proposition}

The proof of this result is given in Section \ref{sec:randomconvergence}. The proof follows the arguments for the analogous result in the graphon setting (see \cite[Lemma 10.16]{Lov12}), with necessary modifications required to ensure that there is a common permutation/partition for which all the layers of $\widebar{\bm{W}}^{\bm{G_n}}$ and $\widebar{\bm W}$ are simultaneously close. 

%
%

\section{Counting Lemma for Multiplexons}
\label{sec:tHW}

One of the fundamental results in graph limit theory is that the map $W \rightarrow t(H, W)$, for $W \in \wsp$, is continuous in the cut distance. This result is a consequence of the counting lemma, which provides a bound on the difference between the homomorphism densities of two graphons in terms of their cut distance (see, for example, \cite[Lemma 10.23]{Lov12} and \cite[Theorem 4.5.1]{zhao2023graph}). The following result establishes the version of this result for multiplexons. The proof relies on a telescoping argument, similar to the one used in the graphon setting. Alternatively, one can also derive this result from the counting lemma for $\wsp$-decorated graphs (see Section \ref{sec:densityprobabilitygraphon} for details).

\begin{lemma}
\label{lem:cont}
For any $r$-multiplex $\bm{H}$, with disjoint decomposition  $\hat {\bm H} = (V(\bm{H}), (\hat H_S)_{S \in \P_0([r])})$,  and any two cumulatively decomposable multiplexons $\widebar{\bm U}, \widebar{\bm W}  \in {\wsp}_{\blacktriangle}^{R-1}$, 
$$\left|t(\hat{\bm{H}},\widebar{\bm{U}}) - t(\hat{\bm{H}},\widebar{\bm{W}})\right| \leq |E(\hat{\bm{H}})| \| \widebar{\bm U} - \widebar{\bm W} \|_{\square} , $$ 
where $E(\hat{\bm{H}}) := \bigcup_{S \in \P_0([r])} E(\hat H_S)$. Consequently, the map ${\wsp}_{\blacktriangle}^{R-1} \ni \widebar{\bm{W}} \mapsto t(\hat{\bm{H}}, \widebar{\bm{W}})$ is continuous in the metric $\delta_\square$  (as defined in \eqref{eq:UWdeltalayer}).
\end{lemma}
\begin{proof} Suppose the vertices of $\bm{H}$ are labeled $V(\bm{H})= \{1, 2, \ldots, |V(\bm{H})|\}$. Note that for each $S \in \P_0([r])$, $\hat H_S = (V(\bm{H}), E(\hat H_S))$ is a simple graph and the collection of edge sets $E(\hat H_S)_{S \in \P_0([r])}$ are disjoint. Then for each edge $e \in E(\hat{\bm{H}})$, there is a unique subset $S_e \in \P_0([r])$ such that $e \in E(\hat H_{S_e})$. 
For any $e = (u, v) \in E(\hat{\bm{H}})$, with $1 \leq u < v \leq |V(\bm H)|$ and a vector $\bm x = (x_1, x_2, \ldots, x_{|V(\bm H)|}) \in [0, 1]^{|V(\bm H)|}$ we denote by $\bm x_e = (x_{u}, x_{v})$. Then, for any two cumulatively decomposable multiplexons $\widebar{\bm U}, \widebar{\bm W}  \in {\wsp}_{\blacktriangle}^{R-1}$, 
\begin{align}
&\left|t(\hat{\bm{H}}, \widebar{\bm{U}}) - t(\hat{\bm{H}}, \widebar{\bm{W}})\right| \nonumber \\
&= \left|\int_{[0,1]^{|V(\bm{H})|}} \left(\prod_{S \in \P_0([r])} \prod_{(i, j) \in E(\hat{H}_S)} \widebar{U}_{S}(x_i,x_j) - \prod_{S \in \P_0([r])} \prod_{(i, j) \in E(\hat{H}_S)} \widebar{W}_{S}(x_i,x_j)\right)\,\mathrm dx_1\,\mathrm dx_2\,\dots\,\mathrm dx_{|V(\bm{H})|}\right| \nonumber \\
&= \left|\int_{[0,1]^{|V(\bm{H})|}} \left(\prod_{ e \in E(\hat{\bm H}) } \widebar{U}_{S_e}( \bm x_{e}) - \prod_{ e \in E(\hat{\bm H}) } \widebar{W}_{S_e}( \bm x_{e} )\right)\,\mathrm dx_1\,\mathrm dx_2\,\dots\,\mathrm dx_{|V(\bm{H})|}\right| , 
\label{eq:differencedecomposable}
\end{align}
where $\bm U = (\widebar{U}_S)_{S \in \P_0([r])}$ and  $\bm W = (\widebar{W}_S)_{S \in \P_0([r])}$ are the  cumulative decompositions of $\bm U$ and $\bm W$, respectively. 

Now, let $\{e_1,\dots, e_{|E(\hat{\bm{H}})|}\}$ be an enumeration of the edges in $\hat{\bm{H}}$. Then using a telescoping sum,
\begin{align*}
& \prod_{ e \in E(\hat{\bm H}) } \widebar{U}_{S_e}( \bm x_{e}) - \prod_{ e \in E(\hat{\bm H}) } \widebar{W}_{S_e}( \bm x_{e} ) \\
&\defeq \sum_{a=1}^{|E(\hat{\bm{H}})|} \left(\prod_{b =1}^{a-1} \widebar{U}_{S_{e_b}}(\bm{x}_{e_b})\right) \left(\prod_{b=a+1}^{|E(\hat{\bm{H}})|} \widebar{W}_{S_{e_b}}(\bm{x}_{e_b}) \right) \left(\widebar{U}_{S_{e_a}}(\bm x_{e_a}) - \widebar{W}_{S_{e_a}}(\bm x_{e_a})\right) .
\end{align*} 
Note that $\prod_{b =1}^{a-1} \widebar{U}_{S_{e_b}}(\bm{x}_{e_b}) \leq 1$ and $\prod_{b=a+1}^{|E(\hat{\bm{H}})|} \widebar{W}_{S_{e_b}}(\bm{x}_{e_b}) \leq 1$. Hence, from \eqref{eq:differencedecomposable}, by the triangle inequality, and \eqref{sqinfmet}, 
\begin{align} 
\left|t(\hat{\bm{H}}, \widebar{\bm{U}}) - t(\hat{\bm{H}}, \widebar{\bm{W}})\right| 
&\leq \sum_{a=1}^{|E(\hat{\bm{H}})|} \| \widebar{U}_{S_{e_a}} - \widebar{W}_{S_{e_a}} \|_{ \square }  \tag*{ (recall \eqref{eq:fgST}) }\nonumber \\  
& \leq |E(\hat{\bm{H}})| \| \widebar{\bm U} - \widebar{\bm W} \|_{\square} , 
\label{eq:tUWsquare} 
\end{align}
where the last step uses \eqref{eq:UWsquarelayer}. 

Since $t(\bm H, \bm W^\sigma) = t(\bm H, \bm W)$, for any $\sigma \in \F$, the continuity of $t(\hat{\bm{H}},\cdot)$ follows from \eqref{eq:tUWsquare} and recalling \eqref{eq:UWdeltalayer}. 
\end{proof}

Next, we establish a converse of the above result for multiplexes. This is referred to as the inverse counting lemma, which shows that if two multiplexons are close in terms of their homomorphism densities, then they are also close in terms of their cut distance. The proof 
is an easy extension of the arguments in the graphon setting \cite[Lemma 10.32]{Lov12}), combined with Proposition \ref{ppn:complete}.

\begin{lemma}
\label{lem:Wcounting} 
Fix $r, K \geq 1$ and two cumulatively decomposable multiplexons $\widebar{\bm U}, \widebar{\bm W}  \in {\wsp}_{\blacktriangle}^{R-1}$. Suppose for any $r$-multiplex $\bm{H}$, with $|V(\bm{H})| \leq K$, the following holds: 
\begin{align}\label{eq:tHUWinverse}
\left|t(\hat{\bm{H}},\widebar{\bm{U}}) - t(\hat{\bm{H}},\widebar{\bm{W}})\right| \leq \frac{1}{2^{rK^2} } ,  
\end{align}
where $\hat {\bm H} = (V(\bm{H}), (\hat H_S)_{S \in \P_0([r])})$ is the disjoint decomposition of $\bm{H}$. Then $$\delta_\square(\widebar{\bm{U}}, \widebar{\bm{W}}) \leq \frac{ 64 R^\frac{3}{2} }{\sqrt{\log K}}.$$  
\end{lemma}

\begin{proof} The condition in \eqref{eq:tHUWinverse} and Lemma \ref{ppn:char} imply that 
$$|\PP(\bm{H} \subseteq \bm{G}(K, \widebar{\bm{U}}) )- \PP(\bm{H} \subseteq \bm{G}(K, \widebar{\bm{W}}) ) | \leq \frac{1}{2^{rK^2} } , $$ 
for any $r$-multiplex $\bm{H}$, with $|V(\bm{H})| = K$. Hence, by the inclusion-exclusion principle, 
$$|\PP(\bm{G}(K, \widebar{\bm{U}}) = \bm{H} ) - \PP( \bm{G}(K, \widebar{\bm{W}}) = \bm{H})| \leq \frac{2^{r {K \choose 2}}}{2^{rK^2} } = \frac{1}{2^{r { {K+1} \choose 2} }} , $$ 
for any $r$-multiplex $\bm{H}$, with $|V(\bm{H})| = K$. Consequently, 
\begin{align} 
\mathsf{TV}\left( \bm{G}(K, \widebar{\bm{U}}) , \bm{G}(K, \widebar{\bm{W}}) \right) 
& = \frac{1}{2} \sum_{\bm{H} : |V(\bm{H})| = K} |\PP(\bm{G}(K, \widebar{\bm{U}}) = \bm{H} ) - \PP( \bm{G}(K, \widebar{\bm{W}}) = \bm{H} ) | \nonumber \\ 
& \leq \frac{2^{r {K \choose 2 } }}{2^{r { {K+1} \choose 2} }}   = \frac{1}{2^{rK}} < 1- 2 e^{-\frac{2 r K}{\log K}} , \nonumber 
\end{align} 
This means we can couple  $\bm{G}(K, \widebar{\bm{U}})$ and $\bm{G}(K, \widebar{\bm{W}})$ such that  $\bm{G}(K, \widebar{\bm{U}}) = \bm{G}(K, \widebar{\bm{W}})$ with probability at least $2 e^{-\frac{2 r K}{ \log K}} $. Now, denote the empirical graphons associated with $\bm{G}(K, \widebar{\bm{U}})$ and $\bm{G}(K, \widebar{\bm{W}})$ by $\widebar{\bm{U}}^{(K)}$ and $\widebar{\bm{W}}^{(K)}$, respectively. Then by Proposition \ref{ppn:complete}, with probability at least $1- 2 e^{-\frac{2 R K}{\log K}}$, 
$$\delta_{\square}(\widebar{\bm{U}}^{(K)}, \widebar{\bm{U}}) \leq \frac{32 R^\frac{3}{2}}{\sqrt{\log K}}  \text{ and } \delta_{\square}(\widebar{\bm{W}}^{(K)}, \widebar{\bm{W}}) \leq \frac{32 R^\frac{3}{2} }{\sqrt{\log K}}.$$
Hence, the above estimates and the coupling such that $\bm{G}(K, \widebar{\bm{U}}) = \bm{G}(K, \widebar{\bm{W}})$ hold with positive probability. This implies, 
$$\delta_{\square}(\widebar{\bm{U}}, \widebar{\bm{W}}) \leq \delta_{\square}(\widebar{\bm{U}}^{(K)}, \widebar{\bm{U}}) + \delta_{\square}(\widebar{\bm{W}}^{(K)}, \widebar{\bm{W}}) \leq \frac{64 R^\frac{3}{2}}{\sqrt{\log K}},$$
which completes the proof of the lemma. 
\end{proof}

Two graphons $U, W \in \wsp$ are said to be {\it weakly isomorphic} if $t(H, U) = t(H, W)$, for all graphs $H$ (without multiple edges or self loops). Analogously, two cumulatively decomposable multiplexons $\widebar{\bm{U}}, \widebar{\bm{W}} \in \wsp^{R-1}_\blacktriangle$ are said to be {\it weakly isomorphic} if $t(\hat{\bm{H}}, \widebar{\bm{U}}) = t(\hat{\bm{H}}, \widebar{\bm{W}})$, for all multiplex networks $\bm{H}$ (where the graphs in each of the layers do not have multiple edges or self loops). The counting lemma and the inverse counting lemma then immediately imply the following result. The analogous result for graphons is given in \cite[Corollary 3.10]{graph_limits_I} (see also \cite[Corollary 10.34]{Lov12}).

\begin{corollary}\label{cor:deltatUW}
Two cumulatively decomposable multiplexons $\widebar{\bm{U}}, \widebar{\bm{W}} \in \wsp^{R-1}_\blacktriangle$ are weakly isomorphic if and only if $\delta_{\square}( \widebar{\bm{U}}, \widebar{\bm{W}} ) = 0$. 
\end{corollary}

In the case of graphons, another equivalent characterization of weak isomorphism is the following (see \cite[Corollary 8.14]{Lov12}): Two graphons $U, W \in \wsp$ are weakly isomorphic if and only if there exist measure preserving transformations $\sigma_1,\sigma_2: [0, 1] \rightarrow [0, 1]$ such that
\[U^{\sigma_1}(x,y) = W^{\sigma_2}(x,y) , \] 
for almost every $(x,y) \in [0,1]^{2}$. In Corollary \ref{cor:UW} we show that an analogous result holds for multiplexons as well. 

\begin{remark}
Corollary \ref{cor:deltatUW} has also been proved for finitely decorated graphs, which covers the case of multiplex networks (see Section \ref{sec:probabilitygraphons}), in \cite[Theorem 5.6]{falgas2016multicolour}. The proof extends Schrijver's argument for the analogous result for graphons (see \cite[Remark 11.4]{Lov12}) to the decorated setting. 
\end{remark}

\section{ Equivalence of Convergence }
\label{left}

We are now ready to prove the main result about the convergence of multiplexes, which establishes the equivalence of left-convergence and convergence in terms of the cut distance.

\begin{theorem}
\label{thm:leftconvergence}
Let $\{\bm{G}_n\}_{ n \geq 1 }$ be a sequence of $r$-multiplexes. Then the following are equivalent: 

\begin{enumerate}
 
\item[$(1)$] $\{\bm{G}_n\}_{n \geq 1}$ is left-convergent, that is, $t(\bm{H},\bm{G}_n)$ converges for all $r$-multiplexes $\bm{H}$. 

\item[$(2)$] There exists a cumulatively decomposable multiplexon $\widebar{\bm{W}} \in {\wsp}_{\blacktriangle}^{R-1}$ such that $\delta_\square(\widebar{\bm{W}}^{\bm{G}_n}, \widebar{\bm{W}} ) \rightarrow 0$, where $\widebar{\bm{W}}^{\bm{G}_n}$ is the empirical multiplexon as in Definition \ref{defn:WGn}.  

\item [$(3)$] There exists a cumulatively decomposable multiplexon $\widebar{\bm{W}} \in {\wsp}_{\blacktriangle}^{R-1}$ such that $t(\bm{H},\bm{G}_n) \to t(\hat{\bm{H}}, \widebar{\bm{W}})$, for every $r$-multiplex $\bm{H}$. 

\end{enumerate}

\end{theorem}

\begin{proof} From Lemma \ref{lem:cont} and Proposition \ref{proposition:gtow}, it is immediate that $(2) \implies (1)$. To establish the reverse implication, consider the sequence $\{\tilde{\widebar{\bm{W}}}^{\bm{G}_n}\}_{n \geq 1}$ of equivalence classes associated with the empirical graphon. Since $({\wsp}_{\blacktriangle}^{R-1})_{\square}$ is a compact space (by Corollary \ref{cumdeccom}), there must exist a subsequence $\{n_s\}_{s \geq 1}$ and a multiplexon $\tilde{\widebar{\bm{W}}}' \in ({\wsp}_{\blacktriangle}^{R-1})_{\square}$ such that 
$$\delta_{\square} ( \widebar{\bm{W}}^{\bm{G}_{n_s}}, \widebar{\bm{W}}') \rightarrow 0 , $$
where $\widebar{\bm{W}}' \in \wsp_{\blacktriangle}^{R-1}$ is any representative element of the equivalence class $\tilde{\widebar{\bm{W}}}'$. 
Hence, by Lemma \ref{lem:cont} and Proposition \ref{proposition:gtow}, $t(\bm{H},\bm{G}_{n_s}) \to t(\hat{\bm{H}}, \widebar{\bm{W}}')$, for every $r$-multiplex $\bm{H}$. Now,  suppose there exists a different convergent subsequence $\{n'_s\}_{s\geq 1}$ such that 
$$\delta_{\square} ( \widebar{\bm{W}}^{\bm{G}_{n_s'}}, \widebar{\bm{W}}'') \rightarrow 0 , $$
for some $\tilde{\widebar{\bm{W}}}'' \in \wsp_{\blacktriangle}^{R-1}$. Then, Lemma \ref{lem:cont} and Proposition \ref{proposition:gtow} again implies, $t(\bm{H},\bm{G}_{n'_s}) \to t(\hat{\bm{H}}, \widebar{\bm{W}}'')$, for all $r$-multiplexes $\bm{H}$. However, because the sequence $\{t(\bm{H},\bm{G}_n)\}_{n \geq 1}$ is convergent, $t(\bm{H},\bm{G}_{n'_s}) \to t(\hat{\bm{H}}, \widebar{\bm{W}}')$, for all $r$-multiplexes $\bm{H}$. Thus, $t(\hat{\bm{H}}, \widebar{\bm{W}}') = t(\hat{\bm{H}}, \widebar{\bm{W}}'')$, for all $r$-multiplexes $\bm{H}$. Hence, by Corollary \ref{cor:deltatUW}, 
$\delta_\square(\widebar{\bm{W}}' , \widebar{\bm{W}}'') = 0$, that is, 
$\tilde{\widebar{\bm{W}}}' = \tilde{\widebar{\bm{W}}}'':=  \tilde{\widebar{\bm{W}}} $. 
Thus, 
$$\delta_{\square} ( \widebar{\bm{W}}^{\bm{G}_{n}}, \widebar{\bm{W}}) \rightarrow 0 , $$ 
for every $\widebar{\bm{W}}\sim\tilde{\widebar{\bm{W}}}$, and $t(\bm{H},\bm{G}_n) \to t(\hat{\bm{H}},\widebar{\bm{W}})$, for every $r$-multiplex $\bm{H}$. This proves $(1) \implies (2)$ and also $(1) \implies (3)$. Since $(3) \implies (1)$ holds trivially, this concludes the proof of the lemma.
\end{proof}

\begin{remark} 
Note that if the hypothesis in Theorem \ref{thm:leftconvergence} (1) holds, then the inverse counting lemma (Lemma \ref{lem:Wcounting}) implies that the sequence $\{\tilde{\widebar{\bm{W}}}^{\bm{G}_n}\}_{n \geq 1}$ is Cauchy in $\delta_\square$. Since $({\wsp}_{\blacktriangle}^{R-1})_{\square}$ is a compact space (by Corollary \ref{cumdeccom}), it is also complete. Hence, there exists $\widebar{\bm{W}} \in {\wsp}_{\blacktriangle}^{R-1}$ such that $\delta_\square(\widebar{\bm{W}}^{\bm{G}_n}, \widebar{\bm{W}} ) \rightarrow 0$. This is an alternate way to show $(1) \implies (2)$ in Theorem \ref{thm:leftconvergence}, that does not use Corollary \ref{cor:deltatUW}. The implication $(2) \implies (3)$ can also be established similarly. 
\end{remark}

\section{ Degree Distribution and Clustering Coefficients }
\label{sec:degreeclusteringcoefficient}

Consider an $r$-multiplex $\bm{G} = (V(\bm G), G_1, G_2, \ldots, G_r)$, with vertex set $V(\bm G) = \{1, 2, \ldots, |V(\bm G)|\}$, and cumulative decomposition $\widebar{\bm G} = (V(\bm{G}), (\widebar G_S)_{S \in \P_0([r])})$. Then, for $i \in V(\bm G)$ and $S \in \P_0([r])$, denote by $d_{\widebar{G}_S}(i)$ the degree of node $i$ in the graph $\widebar{G}_S$. Depending on the
choice of $S$, $d_{\widebar{G}_S}(i)$ has different interpretations in terms of the multiplex layers $G_1, G_2, \ldots, G_r$. 
%
%
For example, when $S= \{s\}$ is a singleton, then $d_{\widebar{G}_{\{s\}}}(i)$ is just the degree of the node $i$ in the graph $G_s$ (the graph in the $s$-the layer). Further, when $S= \{s, t\}$, then $d_{\widebar{G}_{\{s, t\}}}(i)$ is the number
of neighbors of node $i$ that are neighbors in both layers $s$ and $t$, for $1 \leq s \ne t \leq r$. This is a measure of local overlap of the node $i$ between the $s$-th and the $ t$-th layers (see \cite[Section 7.3.1]{bianconi2018multilayer}). 

For each $S \in \P_0([r])$, the {\it empirical degree distribution} of $\widebar{G}_S$ is defined as 
$$D_{\widebar{G}_S} = \frac{1}{|V(\bm G)|}\sum_{i= 1}^{|V(\bm G)|} \delta_{\left(\frac{d_{\widebar{G}_S}(i)}{|V(\bm G)|}\right)},$$
which is the measure on $[0, 1]$ that assigns probability $\frac{1}{|V(\bm G)|}$ to the points $\{\frac{1}{|V(\bm G)|} d_{\widebar{G}_S}(i) \}_{ 1 \leq i \leq |V(\bm G)|}$. The {\it joint empirical degree distribution} of $\bm{G}$ is defined as 
$$D_{\widebar{\bm{G}}} = \frac{1}{|V(\bm G)|}\sum_{i= 1}^{|V(\bm G)|} \delta_{\left(\frac{d_{\widebar{G}_S}(i)}{|V(\bm G)|}\right)_{S \in \P_0([r])}},$$
which is the measure on $[0, 1]^{R-1}$ that assigns probability $\frac{1}{|V(\bm G)|}$ to the collection of vectors $$\left\{ \left( \frac{d_{\widebar{G}_S}(i)}{|V(\bm G)|} \right)_{S \in \P_0([r])} \right\}_{ 1 \leq i \leq |V(\bm G)|}.$$
The following result establishes the convergence of the joint empirical degree distribution: 

\begin{proposition}\label{ppn:degreedistribution}
Consider a sequence of $r$-multiplexes $\{\bm{G}_n\}_{n \geq 1}$ converging to a cumulatively decomposable multiplexon $\widebar{\bm{W}} = (\widebar{W}_S)_{S \in \P_0([r])} \in {\wsp}_{\blacktriangle}^{R-1}$. 
Then 
$$\lim_{n \rightarrow \infty}\mathrm{Wass}(D_{\widebar{\bm{G}}_n}, (d_{\widebar{W}_S}(\eta))_{S \in \P_0([r])}) \rightarrow 0,$$ 
where $d_{\widebar{W}_S}(x) := \int_0^1 \widebar{W}_S(x, y) \mathrm d y$, for $S \in \P_0([r])$, and $\eta \sim \mathrm{Unif}(0, 1)$.\footnote{The Wasserstein distance between 2 measures $\mu$ and $\nu$ on $[0, 1]^d$ is defined as: $\mathrm{Wass}(\mu, \nu) = \inf_{f} | \int f \mathrm d \mu - \int f \mathrm d \nu  | $, where the infimum is taken over all 1-Lipschitz functions $f : [0, 1]^d \rightarrow \mathbb R$. }
\end{proposition}

\begin{proof} 
Consider an 1-Lipschitz function $\bm f: [0, 1]^{R-1} \rightarrow \mathbb R $, that is, $| \bm f (\bm x) - \bm f (\bm y) | \leq \| \bm x - \bm y \|_1$, for $\bm x, \bm y \in [0, 1]^{R-1}$. Suppose $\widebar{\bm{G}}_n = ( (\widebar{G}_n)_S)_{S \in \P_0([r])}$ is the cumulative decomposition of $\bm{G}_n$. Note that, for $S \in \P_0([r])$, 
\begin{align}\label{eq:DW}
& \left| \int \bm f \mathrm d D_{\widebar{\bm{G}}_n}  - \int_0^1 \bm f \left( (d_{\widebar{W}_S}(x))_{S \in \P_0([r])}\right)  \mathrm dx \right| \nonumber \\ 
& = \left| \frac{1}{|V(\bm G_n)|}\sum_{i= 1}^{|V(\bm G_n)|} \bm f \left( \left( \frac{d_{(\widebar{G}_n)_S}(i)}{|V(\bm G_n)|} \right)_{S \in \P_0([r])} \right) - \int_0^1 \bm f  \left( \left( \int_0^1 \widebar{W}_S(x, y) \mathrm d y \right)_{S \in \P_0([r])} \right) \mathrm dx \right| \nonumber \\ 
& = \left|  \int_0^1 \bm f \left( \left( \int_0^1 \widebar{W}^{\bm{G}_n}_S(x, y) \mathrm d y \right)_{ S \in \P_0([r]) } \right) \mathrm dx - \int_0^1 \bm f  \left( \left( \int_0^1 \widebar{W}_S(x, y) \mathrm d y \right)_{S \in \P_0([r])} \right) \mathrm dx \right| \nonumber \\ 
& \leq  \int_0^1 \left|  \bm f \left( \left( \int_0^1 \widebar{W}^{\bm{G}_n}_S(x, y) \mathrm d y \right)_{ S \in \P_0([r]) } \right)  -  \bm f  \left( \left( \int_0^1 \widebar{W}_S(x, y) \mathrm d y \right)_{S \in \P_0([r])} \right) \right| \mathrm dx \nonumber \\ 
& \leq   \sum_{S \in \P_0([r])} \int_0^1 \left| \int_0^1 \left(\widebar{W}^{\bm{G}_n}_S(x, y) - \widebar{W}_S(x, y) \right)  \mathrm d y   \right| \mathrm d x , 
\end{align} 
since $\bm f $ is 1-Lipschitz. Now, define the function $h_S(x) := \bm 1\{\int_0^1 \widebar{W}^{\bm{G}_n}_S(x, y) \mathrm d y \geq \int_0^1 \widebar{W}_S(x, y) \mathrm d y \}$, for $x \in [0, 1]$ and $S \in \P_0([r])$. Then 
\begin{align}\label{eq:dW}
& \int_0^1 \left| \int_0^1 \left(\widebar{W}^{\bm{G}_n}_S(x, y) - \widebar{W}_S(x, y) \right)  \mathrm d y   \right| \mathrm d x \nonumber \\ 
& = \int_{[0, 1]^2}  \left(\widebar{W}^{\bm{G}_n}_S(x, y) - \widebar{W}_S(x, y) \right) h_S(x)  \mathrm d x  \mathrm d y + \int_{[0, 1]^2}  \left( \widebar{W}_S(x, y) - \widebar{W}^{\bm{G}_n}_S(x, y) \right) ( 1 - h_S(x) ) \mathrm d x  \mathrm d y \nonumber \\ 
& \leq 2 \| \widebar{W}^{\bm{G}_n}_S - \widebar{W}_S \|_\square . 
\end{align}
Hence, combining \eqref{eq:DW} and \eqref{eq:dW}, 
$$\mathrm{Wass}(D_{\widebar{\bm{G}}_n}, (d_{\widebar{W}_S}(\eta))_{S \in \P_0([r])}) \leq \sum_{S \in \P_0([r])} 2 \| \widebar{W}^{\bm{G}_n} - \widebar{W}_S \|_\square . $$ 
Taking the infimum of over $\sigma \in \mathcal F$ and noting that the LHS is invariant to measure-preserving bijections,  
$$\mathrm{Wass}(D_{\widebar{\bm{G}}_n}, (d_{\widebar{W}_S}(\eta))_{S \in \P_0([r])}) \leq \inf_{\sigma \in \mathcal F}\sum_{S \in \P_0([r])} 2 \| \widebar{W}^{\bm{G}_n} - \widebar{W}_S^{\sigma} \|_\square =  2\delta_\square (\widebar{\bm{W}}^{\bm{G}_n}, \widebar{\bm{W}}) \rightarrow 0 , $$ 
which completes the proof of the result. 
\end{proof} 

\begin{remark} When $\{G_n\}_{n \geq 1}$ is a sequence of graphs converging to a graphon $W$, Proposition \ref{ppn:degreedistribution} simplifies to $\mathrm{Wass}(D_{G_n}, d_{W}(\eta))  \rightarrow 0$. This is also proved in \cite[Lemma 3.10]{bassino2022random}. 
\end{remark}

In the analysis of social networks, a clustering coefficient is a basic measure of the degree to which nodes in a network tend to cluster together. Specifically, for a graph $G= (V(G), E(G))$, with vertex set $V(G) = \{1, 2, \ldots, |V(G)|\}$ and adjacency matrix $(a_{ij})_{1 \leq i , j \leq n}$, the local clustering coefficient of a node $i \in V(G)$ is defined as (see \cite{watts1998collective} and also \cite[Chapter 10]{networks}): 
\begin{align*} 
\theta_{G}(i) = \frac{ \sum_{1 \leq j \ne k \leq |V(G)|} a_{i j} a_{jk} a_{ki}}{ d_{G}(i) ( d_{G}(i) - 1) }, 
\end{align*}
where $d_{G}(i)$ is the degree of the vertex $i$ in $G$. (Note that $\theta_{G}(i)$ is defined to be zero when $d_G(i)$ is 0 or 1.)  The average of the local clustering coefficients over all the nodes is known as the {\it average clustering coefficient}: 
\begin{align}\label{eq:thetaGaverage}
\widebar{\theta}_{G} := \frac{1}{|V(G)|} \sum_{i=1}^{|V(G)|} \theta_{G}(i), 
\end{align} 
which is an overall measure of the level of clustering in a network. An alternative measure is the global clustering coefficient \cite{luce1949method} (also known as the transitivity ratio \cite[Page 243]{wasserman1994social}): 
\begin{align}\label{eq:tauG}
\tau_{G} = \frac{ \sum_{1 \leq i \ne j \ne k \leq |V(G)|} a_{i j} a_{jk} a_{ki}}{ \frac{1}{2}\sum_{1 \leq i \leq |V(G)|} d_{G}(i) ( d_{G}(i) - 1) } = \frac{6 \times \text{number of triangles}}{ \text{number of 2-stars} }. 
\end{align} 

In a multiplex network, the clustering coefficient can be generalized in several different
ways, since triangles can include links belonging to different layers. Specifically, for an $r$-multiplex $\bm{G} = (V(\bm G), G_1, G_2, \ldots, G_r)$, with vertex set $V(\bm G) = \{1, 2, \ldots, |V(\bm G)|\}$, \cite{structuralmultiplexnetworks} defined local clustering coefficients of a node $i \in V(\bm{G})$ as follows (see also \cite[Section 6.3.1]{bianconi2018multilayer}):  
\begin{align}\label{eq:theta1G}
\theta_{\bm{G}}^{(1)}(i)& = \frac{ \sum_{1 \leq s_1 \ne s_2 \leq r } \sum_{1 \leq j \ne k \leq |V(\bm{G})|} a_{i j}^{(s_1)} a_{jk}^{(s_2)} a_{ki}^{(s_1)}}{ (r-1) \sum_{1 \leq s_1 \leq r } d_{G_{s_1}}(i) ( d_{G_{s_1}}(i) - 1) } ,  \\ 
\theta_{\bm{G}}^{(2)}(i)  & = \frac{ \sum_{1 \leq s_1 \ne s_2 \ne s_3 \leq r } \sum_{1 \leq j \ne k \leq |V(\bm{G})|} a_{i j}^{(s_1)} a_{jk}^{(s_2)} a_{ki}^{(s_3)}}{ (r-2) \sum_{1 \leq s_1 \ne s_3  \leq r } d_{G_{s_1}}(i) d_{G_{s_3}}(i)  } . \label{eq:theta2G}
\end{align}
In words, $\theta_{\bm{G}}^{(1)}(i)$ and $\theta_{\bm{G}}^{(2)}(i)$ are the normalized number of triangles spanning two and three layers that include the node $i \in V(\bm{G})$, respectively. As before, $\theta_1(i)$ and $\theta_2(i)$ are set to zero whenever the respective denominators are zero. As in \eqref{eq:thetaGaverage}, averaging over all the nodes one obtains the average clustering coefficients of $\bm{G}$: 
$$\widebar{\theta}_{\bm{G}}^{(1)} = \frac{1}{|V(G)|} \sum_{i=1}^{|V(G)|} \theta_{\bm{G}}^{(1)}(i)  \text{ and } \widebar{\theta}_{\bm{G}}^{(2)} = \frac{1}{|V(G)|} \sum_{i=1}^{|V(G)|} \theta_{\bm{G}}^{(2)}(i).$$
Alternatively, similar to \eqref{eq:tauG}, one can define the global clustering coefficients (transitivity ratios) of $\bm{G}$ as follows (see \cite[Section 6.3.1]{bianconi2018multilayer}):  
\begin{align*} 
\tau_{\bm{G}}^{(1)} & = \frac{ \sum_{1 \leq s_1 \ne s_2 \leq r } \sum_{1 \leq i \ne j \ne k \leq |V(\bm{G})|} a_{i j}^{(s_1)} a_{jk}^{(s_2)} a_{ki}^{(s_1)}}{ (r-1) \sum_{1 \leq s_1 \leq r } \sum_{1 \leq i \leq |V(\bm{G})|} d_{G_{s_1}}(i) ( d_{G_{s_1}}(i) - 1) } ,  \\ 
\tau_{\bm{G}}^{(2)}  & = \frac{ \sum_{1 \leq s_1 \ne s_2 \ne s_3 \leq r } \sum_{1 \leq i \ne j \ne k \leq |V(\bm{G})|} a_{i j}^{(s_1)} a_{jk}^{(s_2)} a_{ki}^{(s_3)}}{ (r-2) \sum_{1 \leq s_1 \ne s_3 \leq r } \sum_{1 \leq i \leq |V(\bm{G})|} d_{G_{s_1}}(i) d_{G_{s_3}}(i)  } . 
\end{align*}

Since the clustering coefficient of multiplex networks is determined by the counts of triangles and the degree distributions, their limits (as the size of the network grows) can be expressed in terms of the limiting multiplexon. To this end, suppose $\{\bm{G}_n\}_{n \geq 1}$ is a sequence of $r$-multiplexes converging to a cumulatively decomposable multiplexon $\widebar{\bm{W}} = (\widebar{W}_S)_{S \in \P_0([r])} \in {\wsp}_{\blacktriangle}^{R-1}$.  Then, from the definition of left-convergence and Proposition \ref{ppn:degreedistribution}, we get the following convergence results: 

\begin{itemize}

\item Let $$\Theta_{\bm{G}_n}^{(1)} = \frac{1}{|V(\bm{G}_n)|} \sum_{i=1}^{|V(\bm{G}_n)|} \delta_{\theta_{\bm{G}_n}^{(1)}(i)} \text{ and } \Theta_{\bm{G}_n}^{(2)} = \frac{1}{|V(\bm{G}_n)|} \sum_{i=1}^{|V(\bm{G}_n)|} \delta_{\theta_{\bm{G}_n}^{(2)}(i)},$$ be the empirical distributions of the local clustering coefficients defined in \eqref{eq:theta1G} and \eqref{eq:theta2G}, respectively. Then
\begin{align*} 
\Theta_{\bm{G}_n}^{(1)} & \stackrel{D} \rightarrow \frac{ \sum_{1 \leq s_1 \ne s_2 \leq r } \int_{[0, 1]^2} \widebar{W}_{\{s_1\}}(\eta, y) \widebar{W}_{\{s_2\}} (y, z) \widebar{W}_{\{s_1\}} (\eta, z) \mathrm d y \mathrm d z }{ (r-1) \sum_{1 \leq s_1 \leq r } d_{\widebar{W}_{\{s_1\}}} (\eta)^2 } , \nonumber \\ 
\Theta_{\bm{G}_n}^{(2)} & \stackrel{D} \rightarrow \frac{ \sum_{1 \leq s_1 \ne s_2 \ne s_3 \leq r } \int_{[0, 1]^2} \widebar{W}_{\{s_1\}}(\eta, y) \widebar{W}_{\{s_2\}} (y, z) \widebar{W}_{\{s_3\}} (\eta, z) \mathrm d y \mathrm d z }{ (r-2) \sum_{1 \leq s_1 \ne s_3 \leq r } d_{\widebar{W}_{\{s_1\}}} (\eta) d_{\widebar{W}_{\{s_3\}}} (\eta) } , 
\end{align*} 
where $\eta \sim \mathrm{Unif}(0, 1)$.

\item Consequently, 
\begin{align*} 
\widebar{\theta}_{\bm{G}_n}^{(1)} & \rightarrow \int_{[0, 1]} \frac{ \sum_{1 \leq s_1 \ne s_2 \leq r } \int_{[0, 1]^2} \widebar{W}_{\{s_1\}}(x, y) \widebar{W}_{\{s_2\}} (y, z) \widebar{W}_{\{s_1\}} (x, z) \mathrm d y \mathrm d z }{ (r-1) \sum_{1 \leq s_1 \leq r } d_{\widebar{W}_{\{s_1\}}} (x)^2 } \mathrm d x , \nonumber \\ 
\widebar{\theta}_{\bm{G}_n}^{(2)} & \rightarrow \int_{[0, 1]} \frac{ \sum_{1 \leq s_1 \ne s_2 \ne s_3 \leq r } \int_{[0, 1]^2} \widebar{W}_{\{s_1\}}(x, y) \widebar{W}_{\{s_2\}} (y, z) \widebar{W}_{\{s_3\}} (x, z) \mathrm d y \mathrm d z }{ (r-2) \sum_{1 \leq s_1 \ne s_3 \leq r } d_{\widebar{W}_{\{s_1\}}} (x) d_{\widebar{W}_{\{s_3\}}} (x) } \mathrm d x .  
\end{align*} 

\item Moreover, the global clustering coefficients have the following limits: 
\begin{align*}
\tau_{\bm{G}_n}^{(1)} & \rightarrow \frac{ \sum_{1 \leq s_1 \ne s_2 \leq r } \int_{[0, 1]^3} \widebar{W}_{\{s_1\}}(x, y) \widebar{W}_{\{s_2\}} (y, z) \widebar{W}_{\{s_1\}} (x, z) \mathrm d x \mathrm d y \mathrm d z }{ (r-1) \sum_{1 \leq s_1 \leq r } \int_{[0, 1]} d_{\widebar{W}_{\{s_1\}}} (x)^2 \mathrm d x} ,  \\ 
\tau_{\bm{G}_n}^{(2)} & \rightarrow \frac{ \sum_{1 \leq s_1 \ne s_2 \leq r } \int_{[0, 1]^3} \widebar{W}_{\{s_1\}}(x, y) \widebar{W}_{\{s_2\}} (y, z) \widebar{W}_{\{s_3\}} (x, z) \mathrm d x \mathrm d y \mathrm d z }{ (r-2) \sum_{1 \leq s_1 \ne s_2 \leq r } \int_{[0, 1]} d_{\widebar{W}_{\{s_1\}}} (x) d_{\widebar{W}_{\{s_3\}}} (x) \mathrm d x} . 
\end{align*}
\end{itemize}

\section{Examples of Convergent Multiplexes}
\label{sec:examples}

In this section, we discuss various multiplex network models and their multiplexon limits. To begin with, recall the prototypical example of the inhomogeneous random multiplex model from Definition \ref{defn:Wrandomlayers} and Remark \ref{remark:Wrandomlayerscumulative}. This includes, as a special case, the Erd\H{o}s-R\'{e}nyi random multiplex model (recall Example \ref{example:random12} and Remark \ref{remark:Gnpparameters}). Another relevant special case of the inhomogeneous random multiplex model is the correlated stochastic block model, where the graphon in each layer is a block function. 
 
\begin{example}[Correlated Stochastic Block Model]  
\label{example:blockgraphon}
Fix $r \geq 1$. An $r$-multiplex $\bm G_n$ with $V(\bm{G}_n) = [n]$ and disjoint decomposition $([n], ((\hat{G}_n)_S)_{S \in \P([r])})$ is distributed according to a correlated stochastic block model (CSBM) with $K$-blocks, if it is are generated as follows: Suppose $z_i$ is the block-label of the vertex $i \in [n]$. Then, independently for $1 \leq i < j \leq n$, 
\begin{align}\label{eq:EGab}
\mathbb{P}((i, j) \in E((\hat{G}_n)_{S}) \mid z_i = a, z_j = b)) &= \hat{p}_S^{\{a, b\}} , 
\end{align}
where, for each $a, b \in [K]$, $(\hat{p}_S^{\{a, b\}})_{S \in \P([r])}$ are non-negative constants satisfying $\sum_{S \in \P([r])} \hat{p}_S^{\{a, b\}} = 1$.  
The label vector $(z_1, z_2, \ldots, z_n)$ is usually assumed to follow $\mathrm{Multi}(K, q_1, q_2, \ldots, q_K)$, where $\sum_{a=1}^K q_a = 1$. Then the multiplex $\bm G_n$ converges to the disjointly decomposable multiplexon $\hat{\bm{W}} = (\hat{W}_S)_{S \in \P([r])}$ defined as: 
$$\hat{W}_S (x, y) =  \hat{p}_S^{\{a, b\}}  , \text{ for } (x, y) \in \left [\sum_{s=1}^{a-1} q_s, \sum_{s=1}^{a} q_s \right] \times \left [\sum_{s=1}^{b-1} q_s, \sum_{s=1}^{b} q_s \right], $$
for $S \in \P([r])$. 
Note that when $r=1$, this reduces to the classical (single-layer) stochastic block model, which is one of the most widely studied models in network analysis \cite{holland1983stochastic,olhede2014network,kolaczyk2014statistical}. 
Recently, CSBM has emerged as an important prototype for modeling multilayer social network data \cite{pamfil2020inference,zhang2024consistent,barbillon2017stochastic} and in graph matching problems (see \cite{correlatedsbm,racz2021correlated,gaudio2022exact,racz2023matching,lyzinski2018information,lyzinski2015spectral,onaran2016optimal} and references therein). 
The CSBM can be parameterized in a number of different ways. For example, similar to the Erd\H{o}s-R\'{e}nyi multiplex (recall Remark \ref{remark:Gnpparameters}), the cumulative decomposition of the CSBM can be obtained using the relations  \eqref{eq:cd} and \eqref{eq:dc}. For alternative parameterization, see \cite[Section 5.1]{chandna2022edge}. 
In applications, the probabilities in \eqref{eq:EGab} are often assumed to have some parametric forms. For instance, Barbillon et al. \cite{barbillon2017stochastic} considered a logistic model for \eqref{eq:EGab} in terms of observed edge covariates, and Zhang et al. \cite{zhang2024consistent} expressed the probabilities in \eqref{eq:EGab} using an Ising model. 
\end{example}

Next, we consider a multiplex analogue of threshold graphs \cite[Example 11.36]{Lov12} (see also \cite{thresholdgraphs}). For simplicity, we describe the model with two layers. 

\begin{example}[Multiplex Threshold Graphs] Consider the 2-multiplex $\bm{G}_n = (V(\bm{G}_n), (G_n)_1, (G_n)_2)$, where $V(\bm{G}_n) = \{1, 2, \ldots, n\}$ and 
$$E((G_n)_1) := \{(i, j): i+j \leq \lceil a n \rceil \} \text{ and } E((G_n)_2) := \{(i, j): i+j \leq \lceil b n \rceil \} ,$$
for some fixed $a, b \in [0, 1]$. This multiplex converges to a cumulatively decomposable multiplexon $\widebar{\bm{W}} = (\widebar{W}_{\{1\}}, \widebar{W}_{\{2\}}, \widebar{W}_{\{1, 2\}} )$, where 
$$\widebar{W}_{\{1\}} (x, y) = \bm 1\{x+ y \leq a\}, \quad \widebar{W}_{\{2\}} (x, y) = \bm 1\{x+ y \leq b\},$$ 
and $$\widebar{W}_{\{1, 2\}} (x, y) = \bm 1\{x+ y \leq \min\{a, b\} \}.$$
\end{example}

Next, we consider a multiplex version of the uniform attachment model in \cite[Example 11.39]{Lov12}. As before, we consider two layers.

\begin{example}[Multiplex Uniform Attachment]
\label{ex:unifattach} 
Suppose $\bm{G}_n = (V(\bm{G}_n), (G_n)_1, (G_n)_2)$ is a 2-layer multiplex with $V(\bm{G}_n) = \{0, 1, \dots,n-1\}$ where the edge sets are generated sequentially as follows: Start with a single node labeled 0. At the $n$-th iteration, for $n \geq 2$, a new node labeled $n-1$ is born, and then edges are added to the multiplex according to the following rules: 

\begin{itemize}

\item For the first layer, every pair of nonadjacent edges is connected with probability $\frac{1}{n}$. 

\item For the second layer, if $(i, j) \notin E((G_{n})_1)$, then $(i, j) \notin E((G_{n})_2)$. Otherwise, $(i, j) \in E((G_{n})_2)$ with probability $\frac{1}{n}$ (independent of any other edge). 

\end{itemize}
Now, following along with \cite[Example 11.39]{Lov12}, we can do some simple computations. Suppose that $i < j< n$. Then 
\begin{align}\label{eq:ijn}
\PP\left((i, j) \notin E((G_{n})_1)\right) = \frac{j}{j+1}\frac{j+1}{j+2}\cdots\frac{n-1}{n} = \frac{j}{n} . 
\end{align}
Hence,  $\PP\left((i, j) \in E((G_{n})_1)\right) = 1 - \frac{j}{n}$. Similarly, for $j < m \leq n$,
\begin{align}\label{eq:ijmn}
\PP\left((i, j) \notin E((G_{n})_2)\middle|(i, j) \in E((G_{m})_1)\setminus E( (G_{m-1})_1)\right) &= \frac{m-1}{m}\frac{m}{m+1}\cdots\frac{n-1}{n} = \frac{m-1}{n} , 
\end{align} 
and 
\begin{align}\label{eq:ijm} 
\PP\left((i, j) \in E((G_{m})_1)\setminus E( (G_{m-1})_1\right) &= \begin{cases}
\frac{j}{m-1}\cdot\frac{1}{m} &\te{ if } m > j+1,\\
\frac{1}{j+1} &\te{ if } m = j+1,
\end{cases} \nonumber \\
&= \frac{j}{m(m-1)} .  
\end{align}
Combining \eqref{eq:ijmn} and \eqref{eq:ijm} gives, 
\begin{align*}
\PP\left((i, j) \notin E((G_{n})_2)\te{ and } (i, j) \in E((G_{m})_1)\setminus E( (G_{m-1})_1\right) &= \frac{j}{mn}. 
\end{align*}
This implies, 
\begin{align*}
& \PP\left((i, j) \notin E((G_{n})_2)\te{ and } (i, j)\in E((G_{n})_1)\right) \nonumber \\ 
& = \sum_{k=j+1}^n \PP\left((i, j) \notin E((G_{n})_2)\te{ and } (i, j) \in E((G_{k})_1)\setminus E( (G_{k-1})_1\right) \nonumber \\ 
& = \frac{j}{n}\sum_{k=j+1}^n \frac{1}{k}. 
\end{align*} 
Furthermore, from \eqref{eq:ijn}, 
\begin{align*} 
\PP\left((i, j) \notin E((G_{n})_2)\te{ and } (i, j)\notin E((G_{n})_1)\right) &= \PP\left((i, j)\notin E((G_{n})_1)\right) = \frac{j}{n}. 
\end{align*} 
Hence, combining the above two equations, 
\begin{align*}
\PP\left((i, j) \notin E((G_{n})_2)\right) &= \frac{j}{n}\left(1 + \sum_{k=j+1}^n \frac{1}{k}\right).
\end{align*} 
Note that for $z \in (0,1]$,
\[\lim_{n\to\infty} \sum_{k=\lceil zn\rceil+1}^n \frac{1}{k} = \lim_{n\to\infty} \int_{zn+1}^n \frac{1}{t}\,dt = \lim_{n\to\infty}\ln\left(\frac{n}{zn+1}\right) = -\ln(z).\]
Note that substituting $i = \lceil xn\rceil$ and $j = \lceil yn \rceil$ for $(x,y) \in [0,1]^2$, in the limit $\frac{j}{n}$ can be replaced by $\lim_{n\to\infty}\lceil n\max\{x,y\} \rceil/n = \max\{x,y\}$ in the above equations. Hence, denoting $\widebar{\bm{W}}^{G_n} = (\widebar{W}^{G_n}_{\{1\}}, \widebar{W}^{G_n}_{\{2\}}, \widebar{W}^{G_n}_{\{1, 2\}})$ as the empirical multiplexon for $\bm{G}_n$, we have 
\begin{align*}
\lim_{n\to\infty} \PP\left(\widebar{W}^{G_n}_{\{1\}}(x,y) = 1\right) &= 1 - \max\{x,y\},\\
\lim_{n\to\infty} \PP\left(\widebar{W}^{G_n}_{\{2\}}(x,y) = 1\right) &= 1 - \max\{x,y\}\left(1 - \ln\left(\max\{x,y\}\right)\right),\\
\lim_{n\to\infty} \PP\left(\widebar{W}^{G_n}_{\{1,2\}}(x,y) = 1\right) &= 1 - \max\{x,y\}\left(1 - \ln\left(\max\{x,y\}\right)\right).
\end{align*} 
Hence, $\widebar{\bm{W}}^{G_n} \to \widebar{\bm{W}} = (\widebar{W}_{\{1\}}, \widebar{W}_{\{2\}}, \widebar{W}_{\{1, 2\}})$ almost everywhere and therefore in $\wsp^R_{\blacktriangle}.$
\[\widebar{W}_S(x,y) = \begin{cases}
1 - \max\{x,y\} &\te{ if } S = \{1\},\\
1 - \max\{x,y\}\left(1 - \ln\left(\max\{x,y\}\right)\right) &\te{ if } S = \{2\}\te{ or } \{1,2\}.
\end{cases}\]
\end{example}

Another class of examples that can be considered as multiplex networks are dynamic network models. In the following, we present a simple example (see \cite{dynamicfeedback} for more general settings where the evolution of the graph structure also depends on an underlying latent process). 

\begin{example}[Dynamic Graphs]
For each $n$, let $\bm{G}_n = (G_n(t))_{t \geq 0}$ be a dynamic graph sequence with fixed vertex set $V(\bm{G}_n) = [n]$, where $\bm{G}_n(0) \sim G(n,p)$ (for some $p \in [0, 1]$) and at each subsequent integer time $t \geq 1$, the graph $G_n(t)$ is obtained from $G_n(t-1)$ as follows:\footnote{A similar model in continuous time appears in \cite{braunsteins2023sample}, where large deviations of sample paths are established (see also the recent paper \cite{bhamidi2025large} for a more general dynamical model on directed graphs).} 
Independently for any $1 \leq i < j \leq n$: 
$$\PP\left((i,j) \in E(\bm{G}_n(t))\middle|  (i,j) \in E(\bm{G}_n(t-1)) \right) = q_1$$
and 
$$\PP\left((i,j) \in E(\bm{G}_n(t))\middle|  (i,j) \notin E(\bm{G}_n(t-1)) \right) = 1-q_0, $$
for $q_0, q_1 \in [0, 1]$. In other words, independently each $1 \leq i < j \leq n$, the process $\{X_{ij}^{(t)}\}_{t \geq 0}$, where $X_{ij}^{(t)}= \bm 1\{(i,j) \in E(\bm{G}_n(t))\}$, evolves as an $\{0, 1\}$-valued Markov chain with transition matrix 
$$
\bm Q= \begin{pmatrix} 
q_0 & 1-q_0 \\ 
1-q_1 & q_1
\end{pmatrix} . 
$$ 
Now, fix a finite time horizon $T \geq 1$ and denote by $[T]_0:= \{0, 1, \ldots, T\}$. Then consider the graph trajectory $\bm{G}_n([T]_0) := (G_n(0), G_n(1), \ldots G_n(T))$ as a $(T+1)$-multiplex on the vertex set $[n]$. Then, as $n \rightarrow \infty$, $\bm{G}_n([T]_0)$ converges to a cumulatively decomposable $(2^{T+1}-1)$-multiplexon $ \widebar{\bm{W}} = (\widebar{W}_S)_{S \in \P_0([T]_0)}$, where 
\begin{align}\label{eq:dynamic}
\widebar W_S(x, y) = \PP(X_{12}^{(s)} = 1, \text{ for all } s \in S), 
\end{align}
for $S \in \P_0([T]_0)$ and all $(x, y) \in [0, 1]$. For each layer,  $S \in \P_0([T]_0)$, the function $\widebar W_S$ is a constant function given by the probability in \eqref{eq:dynamic}, which can be computed in terms of $p, q_1, q_2$ using the Markov property. For example, when $S = \{t\}$, for $0 \leq t \leq T$, is a singleton, then $\widebar{W}_{\{0\}} (x, y) = p$, for all $x, y \in [0, 1]$, and for $t \geq 1$, 
$$\widebar{W}_{\{t\}} (x, y) = p(\bm Q^t)_{22} + (1-p) (\bm Q^t)_{12},$$ for all $(x, y) \in [0, 1]$, where $\bm Q^t$ is the $t$-th power of the transition matrix $\bm Q$ and $(\bm Q^t)_{ij}$ is the $(i, j)$-th element of the matrix $\bm Q^t$, for $1 \leq i , j \leq 2$ (the computation of $\bm{Q}^t$ is given in \cite[Section 5.1]{taylorintroduction}). Note that $\widebar{\bm{W}}$ is an example of a correlated Erd\H{o}s-R\'enyi multiplex.
\end{example}

\section{Connections to Decorated/Probability Graphons} 
\label{sec:probabilitygraphons}

Lov\'asz and Szegedy \cite{lovasz2010limits} introduced the notion of measure-valued graphons as limit objects for decorated graphs, which, in particular, covers the case of multiplex networks. Specifically, given a Polish space $\bm{Z}$, a $\bm{Z}$-decoration of a simple graph $F= (V(F), E(F))$ is a map $g: E(F) \rightarrow \bm{Z}$. The limit of a sequence of $\bm{Z}$-decorated graphs can be described in terms of a {\it probability-graphon} \cite{abraham2025probability}, which is a symmetric probability kernel 
\begin{align}\label{eq:WZm}
\bm{W} : [0, 1]^2 \rightarrow \mathcal{M}(\bm{Z}) , 
\end{align}
where $\mathcal{M}(\bm{Z})$ is the space of probability measures on $\bm{Z}$. Here, symmetric means $\bm{W}(x, y; A) = \bm{W}(y, x; A)$ for all $x, y \in [0, 1]$ and for every measurable subset $A$ of $\bm{Z}$. One can interpret a probability graphon as follows: for any two `vertices' $x, y \in [0, 1]$, the weight of the edge between $x$ and $y$ is distributed as the probability measure $\bm{W}(x, y; \mathrm dz)$. Lov\'asz and Szegedy \cite{lovasz2010limits} introduced the notion of left-convergence of $\bm{Z}$-decorated graphs when is $\bm{Z}$ is compact and referred to functions of the form \eqref{eq:WZm} as $\bm{Z}$-graphons (see also \cite[Chapter 17]{Lov12}). The compactness assumption was later relaxed to include general Banach space decorations in \citet{KunLovSze22}. Convergence properties of decorated graphs when $\bm{Z}$ is finite or countable, which includes the case of multigraphs and multiplex networks, were considered in \cite{kolossvary2011multigraph,falgas2016multicolour,rollin2023dense}. 
Recently, \citet{abraham2025probability} established metric properties of the convergence of decorated graphs when $\bm{Z}$ is a general Polish space and also introduced the term ``probability graphon'' (see also \cite{zucal2024probabilitygraphons} for an equivalent formulation in terms of $P$-variables). The general framework of probability graphons includes both graphons and multiplexons: 

\begin{enumerate}[label=\textrm{(S\arabic*)}]

\item \label{itm:Z} When $\bm{Z} = \{ 0, 1\}$, then a probability-graphon can be equivalently represented in terms of a graphon. This is because any probability-graphon $\bm{W}$ on $\bm{Z}$ has the form
\begin{align}\label{eq:WZ}
\bm{W} (x,y, \cdot) =  (1-W(x, y)) \delta_0 + W(x, y) \delta_1, 
\end{align}
for $x, y \in [0, 1]$ and some graphon $W: [0, 1]^2 \rightarrow [0, 1]$ and vice versa. Moreover, any simple graph $G= (V(G), E(G))$ can be realized as a $\bm Z$-decoration of the complete graph on $|V(G)|$ vertices, where $g((i, j))= \bm 1\{(i,j) \in E(G)\}$, for $1 \leq i < j \leq |V(G)|$. 

\item \label{itm:ZR} More generally, for $r \geq 1$ and $\bm{Z} = \P([r])$, a probability-graphon can be represented as follows: For $x, y \in [0, 1]$, 
\begin{align}\label{eq:WZS}
\bm{W} (x, y, \cdot) = \sum_{S \in \P([r])} \hat{W}_S(x, y) \delta_S, 
\end{align}
where $\hat{W}_S: [0, 1]^2 \rightarrow [0,1]$ are graphons such that $\sum_{S \in \P([r])} \hat{W}_S(x, y) = 1$, for almost every $x, y \in [0, 1]$. Hence, in this case, any probability-graphon $\bm{W}$ can be equivalently represented as a disjointly decomposable multiplexon $\hat{\bm{W}} = (\hat{W}_S)_{S \in \P([r])}$. Moreover, any $r$-multiplex $\bm{G}= (V(\bm{G}), G_1, G_2, \ldots, G_r)$ can be realized as a $\bm Z$-decoration of the complete graph on $|V(\bm{G})|$ vertices, where 
$$g((i, j))= \{ s \in [r]: (i, j) \in E(G_s) \},$$ for $1 \leq i < j \leq |V(\bm G)|$. In other words, each $(i, j)$ is decorated with the set that consists of the layers of $\bm{G}$ where $(i, j)$ is an edge (recall also the notion of multilinks from Remark \ref{remark:mlink}). If we define $G_S^g = (V(\bm G), E(G_S^g)) $ with $E(G_S^g) = \{ 1 \leq i < j \leq |V(\bm G)|: g((i, j)) = S\}$, then $(V(\bm G), (G_S^g)_{S \in \P_0([r])})$ is precisely the disjoint decomposition of $\bm{G}$.

\end{enumerate}

\subsection{Cut Distance for Probability Graphons}

Denote by $$\bm{\wsp}_{\bm Z} := \{ \bm{W} : [0, 1]^2 \rightarrow \mathcal{M}(\bm{Z}) \text{ such that }  \bm{W} \text{ is a probability graphon} \},$$ the collection of probability graphons on $\bm{Z}$. Then given some metric $\mathsf{d}(\cdot, \cdot)$ on $\M(\bm{Z})$ and two probabilty graphons $\bm{U}, \bm{W} \in \bm{\wsp}_{\bm Z}$, \citet{abraham2025probability} defined the cut-norm of $\bm{U} - \bm {W}$ as: 
\begin{align}\label{eq:WST} 
\| \bm{U} - \bm {W}\|_{\square, \mathsf{d}} = \sup_{A, B \subseteq [0, 1]} \mathsf{d} \left( \int_{A \times B}\bm{U}(x, y; \cdot ) \mathrm d x \mathrm d y , \int_{A \times B}\bm{W}(x, y; \cdot ) \mathrm d x \mathrm d y \right) . 
\end{align}  
(A similar metric is also defined in \cite{athreya2023path}). 
Then the cut distance between $\bm{U}$ and $\bm {W}$ is defined as: 
\begin{align}\label{eq:WSTdistance} 
\delta_{\square, \mathsf{d}}  (\bm{U}, \bm {W})= \inf_{\sigma \in \mathcal{F}}\| \bm{U} - \bm {W}^{\sigma}\|_{\square, \mathsf{d}} , 
\end{align}  
where $\bm{W}^{\sigma}(x, y, \cdot) = \bm{W}(\sigma(x), \sigma(y); \cdot )$, for $x, y \in [0, 1]$. Denote by $(\bm{\wsp}_{\bm Z})_{\square, \mathsf{d}}$ the space of probability graphons on $\bm{Z}$ where any pair of probability graphons $\bm{U}, \bm {W} \in \bm{\wsp}_{\bm Z}$ with $\delta_{\square, \mathsf{d}}  (\bm{U}, \bm {W}) = 0$ are considered the same point in the space. Common choices of $\mathsf{d}$ are the L\'evy-Prokhorov ($\mathsf{LP}$) metric, the Kantorovitch-Rubinstein ($\mathsf{KR}$) metric, and the Fortet-Mourier metric, all of which induce the same topology on $(\bm{\wsp}_{\bm Z})_{\square, \mathsf{d}}$ (see \cite[Proposition 1.1]{abraham2025probability}). Moreover, when $\bm{Z} = \P([r])$ as in setting \ref{itm:ZR} above, then \eqref{eq:WSTdistance} induces the same topology as the cut distance $\delta_{\square}$ defined in \eqref{eq:UWdeltalayer}. Specifically, we have the following relation (the proof is given in Appendix \ref{sec:UWlayersmetricpf}):

\begin{proposition}\label{ppn:UWlayersWp}
Suppose $\bm{U} (x, y, \cdot) = \sum_{S \in \P([r])} \hat{U}_S(x, y) \delta_S, $ and $\bm{W} (x, y, \cdot) = \sum_{S \in \P([r])} \hat{W}_S(x, y) \delta_S$ are two probability graphons on $\bm{Z} = \P_0([r])$ as in \eqref{eq:WZS}. Then 
\begin{align}\label{eq:UWlayersWp}
 2 \delta_{\square, \mathsf{LP}} (\bm{U}, \bm {W}) \leq \delta_{\square}  (\hat{\bm{U}}, \hat{\bm {W}}) \leq  3 R \delta_{\square, \mathsf{LP}} (\bm{U}, \bm {W}), 
\end{align} 
where $R= 2^{r}$, $\delta_{\square}$ is as defined in \eqref{eq:UWdeltalayer}, and $\hat{\bm{U}} = (\hat{U}_S)_{S \in \P([r])}$ and $\hat{\bm{W}} = (\hat{W}_S)_{S \in \P([r])}$ are disjointly decomposable multiplexons.  
\end{proposition}

\begin{remark} 
Proposition \ref{ppn:UWlayersWp} shows that the topologies generated by $\delta_\square$ and $\delta_{\square, \mathsf{LP}}$ are equivalent. Hence, as mentioned in Remark \ref{remark:compactpf}, one can recover the compactness result in Theorem \ref{cumdeccom} by invoking \cite[Theorem 1.3]{abraham2025probability}. 
\end{remark}

\subsection{ Homomorphism Densities for Probability Graphons }
\label{sec:densityprobabilitygraphon}

In the general framework of probability graphons, homomorphism densities are defined in terms of edge-decorated graphs. Towards this, denote by $C_b(\bm Z)$ the collection of all bounded continuous functions from $\bm Z$ to $\mathbb \R$. Given a finite graph
$H = (V(H), E(H))$ a $C_b(\bm Z)$-decoration of $H$ assigns to every edge $e \in E(H)$ a function $g_e \in C_b(\bm Z)$. The decorated graph will be denoted by $H^{\bm{g}}$, where $\bm g = \{g_e\}_{e \in E(H)}$. Then the homomorphism density of the decorated graph $H^{\bm{g}}$ in a probability graphon $\bm{W}$ on $\bm{Z}$ is defined as: 
\begin{align}\label{eq:tHWprobabilitygraphon}
 t(H^{\bm{g}}, \bm{W}) = \int_{[0, 1]^{V(H)}} \prod_{(i, j) \in E(H)} \left( \int_{\bm Z} g_{(i, j)}(z) \bm{W}(x_i, x_i; \mathrm d z) \right) \mathrm dx_1\, \mathrm dx_2\,\dots\, \mathrm dx_{|V(H)|} . 
\end{align} 
This reduces to the notions of homomorphism densities of graphons and multiplexons when $\bm{Z}$ is a finite set and the decorations are defined appropriately (recall \ref{itm:Z} and \ref{itm:ZR}): 

\begin{itemize}

\item Suppose $\bm{Z}= \{0, 1\}$ and $H = (V(H), E(H))$  is a finite simple graph. Then considering the trivial decoration $g_{(i, j)}(z) = \bm 1\{z=1\}$, for $z \in \{0, 1\}$ and $(i, j) \in E(H)$, and recalling the representation in \eqref{eq:WZ}, the expression in \eqref{eq:tHWprobabilitygraphon} reduces to
\begin{align*} 
 t(H^{\bm{g}}, \bm{W}) = \int_{[0, 1]^{V(H)}} \prod_{(i, j) \in E(H)} W(x_i, x_i) \mathrm dx_1\, \mathrm dx_2\,\dots\, \mathrm dx_{|V(H)|} = t(H, W) . 
\end{align*}

\item Next, suppose $\bm{Z} = \P([r])$ and $\bm{H} = (V(\bm{H}), E(H_1), E(H_2), \ldots, E(H_r))$ is an $r$-multiplex. Let $\hat{\bm{H}} = (V(\bm H), (\hat{H}_S)_{S \in \P_0([r])})$ be the disjoint decomposition of $\bm{H}$ and $\widebar{\bm{H}} = (V(\bm H), (\widebar{H}_S)_{S \in \P_0([r])})$ the cumulative decomposition of $\bm{H}$. Note that for each $S \in \P_0([r])$, $\hat H_S = (V(\bm{H}), E(\hat H_S))$ is a simple graph and the collection of edge sets $E(\hat H_S)_{S \in \P_0([r])}$ are disjoint. Hence, for each edge $(i, j) \in E(\hat{\bm{H}}) := \bigcup_{S \in \P_0([r])} E(\hat H_S)$, there is a unique subset $S_{(i, j)} \in \P_0([r])$ such $(i, j) \in E(\hat{H}_{S_{(i, j)}})$. Now, consider the simple graph $\hat H := (V(\bm{H}), E(\hat{\bm{H}}))$ decorated as follows:  For $(i, j) \in E(\hat{\bm{H}})$, 
$$g_{(i, j)}(S) = \bm 1\{ (i, j) \in E(\widebar H_S)\} = \bm 1\{ S \supseteq S_{(i, j)} \},$$  when $S \ne \emptyset$ and $g_{(i, j)}(\emptyset) = 0$. Then, for any probability graphon $\bm{W}$, recalling the representation in \eqref{eq:WZS}, the expression in \eqref{eq:tHWprobabilitygraphon} reduces to
\begin{align}
 t(\hat{H}^{\bm{g}}, \bm{W}) 
& =  \int_{[0, 1]^{|V(\bm{H})|}} \prod_{(i, j) \in E(\hat{\bm{H}})} \left( \sum_{ S' \in \P_0([r])} g_{(i, j)}(S') \hat{W}_{S'}(x_i, x_j) \right) \mathrm dx_1\, \mathrm dx_2\,\dots\, \mathrm dx_{|V(\bm{H})|}  \nonumber \\  
& =  \int_{[0, 1]^{|V(\bm{H})|}} \prod_{(i, j) \in E(\hat{\bm{H}})} \left( \sum_{ S' \supseteq S_{(i, j)}} \hat{W}_{S'}(x_i, x_j) \right) \mathrm dx_1\, \mathrm dx_2\,\dots\, \mathrm dx_{|V(\bm{H})|}  \nonumber \\  
& =  \int_{[0, 1]^{|V(\bm{H})|}} \prod_{(i, j) \in E(\hat{\bm{H}})} \widebar{W}_{S_{(i, j)}}(x_i, x_j) \mathrm dx_1\, \mathrm dx_2\,\dots\, \mathrm dx_{|V(\bm{H})|} \label{eq:tHgW}  \\  
& = \int_{[0, 1]^{|V(\bm{H})|}} \prod_{S \in \P_0([r])} \prod_{(i, j) \in E(\hat{H}_S)} \widebar{W}_{S}(x_i, x_i) \mathrm dx_1\, \mathrm dx_2\,\dots\, \mathrm dx_{|V(\bm{H})|} \nonumber \\ 
& = t(\hat{\bm{H}}, \widebar{\bm{W}}) , \nonumber 
\end{align} 
\end{itemize}
where \eqref{eq:tHgW} uses \eqref{eq:dc} and the last equality uses \eqref{eq:tHW}. 

Using the above representations we can recover the counting lemma of multiplexons (recall Lemma \ref{lem:cont}) from the counting for $\wsp$-decorated graphs (\cite[Lemma 10.24]{Lov12}).\footnote{A simple graph is $H=(V(H), E(H))$ is said to be $\wsp$-decorated, if every edge $e \in E(H)$ is assigned a graphon $W_{e} \in \wsp$. } Specifically, given any cumulatively decomposable multiplexon 
$\widebar{\bm{W}} \in \wsp^{R-1}_\blacktriangle$, assign to the edge $(i, j)$ of the graph $\hat H := (V(\bm{H}), E(\hat{\bm{H}}))$ the graphon $\widebar{W}_{S_{(i, j)}}$, 
where $S_{(i, j)}$ is defined above. Hence, the collection $\{\widebar{W}_{S_{(i, j)}}: (i, j) \in E(\hat{\bm{H}})\}$ is a $\wsp$-decoration of $\hat H$. Hence, from \eqref{eq:tHgW} and \cite[Lemma 10.24]{Lov12}, for any two cumulatively decomposable multiplexons $\widebar{\bm{U}}, \widebar{\bm{W}} \in \wsp^{R-1}_\blacktriangle$, we have 
\begin{align}
\left|t(\hat{\bm{H}},\widebar{\bm{U}}) - t(\hat{\bm{H}},\widebar{\bm{W}})\right| & = \left| t(\hat{H}^{\bm{g}}, \bm{U})  -  t(\hat{H}^{\bm{g}}, \bm{W})  \right| \nonumber \\ 
& = \sum_{(i, j) \in E(\hat{H})} \left\| \widebar{U}_{S_{(i, j)}} - \widebar{W}_{S_{(i, j)}} \right\|_\square \nonumber \\ 
& \leq |E(\hat{H})| \left\| \widebar{\bm{U}} - \widebar{\bm{W}} \right\|_\square =  |E(\hat{\bm{H}})| \left\| \widebar{\bm{U}} - \widebar{\bm{W}} \right\|_\square , \nonumber 
\end{align} 
which gives the result in Lemma \ref{lem:cont}.



\subsection{ Weak Isomorphism }
\label{sec:deltaUW}

For probability graphons, \citet{abraham2025probability} defines weak isomorphism as follows (see \cite[Definition 3.16]{abraham2025probability}): Two probability graphons $\bm{U}, \bm{W} \in \wsp_{\bm Z}$ are said to be {\it weakly isomorphic} if there exist measure preserving transformations $\sigma_1,\sigma_2:  [0, 1] \rightarrow [0, 1]$ such that
\[\bm{U}^{\sigma_1}(x,y, \cdot) = \bm{W}^{\sigma_2}(x,y, \cdot) , \] 
for almost every $(x,y) \in [0,1]^{2}$. In \cite[Theorem 3.17]{abraham2025probability} it is shown that $\bm{U}, \bm{W} \in \wsp_{\bm Z}$ are weakly isomorphic if and only if $\delta_{\square, \textsf{d}} (\bm{U},  \bm{W}) = 0$, whenever the distance $\mathrm{d}$ on $\mathcal M(\bm{Z})$ is invariant and smooth (see \cite[Definition 3.10]{abraham2025probability}). Hence, when $\bm Z= \P_0([r])$ as in setting \ref{itm:ZR}, the equivalence of the metrics $\delta_{\square, \textsf{LP}}$ and $\delta_\square$ implies the following result:

\begin{corollary}\label{cor:UW}
Suppose $\widebar{\bm{U}}, \widebar{\bm{W}} \in \wsp^{R-1}_\blacktriangle$ are two cumulatively decomposable multiplexons. Then the following are equivalent: 

\begin{enumerate}

\item[(1)] $t(\hat{\bm{H}}, \widebar{\bm{U}}) = t(\hat{\bm{H}}, \widebar{\bm{W}})$, for all simple multiplex network $\bm{H}$;

\item[(2)] $\delta_{\square}( \widebar{\bm{U}}, \widebar{\bm{W}} ) = 0$; 

\item[(3)] there exist measure preserving transformations $\sigma_1,\sigma_2:  [0, 1] \rightarrow [0, 1]$ such that
$\widebar{\bm{U}}^{\sigma_1}(x,y) = \widebar{\bm{W}}^{\sigma_2}(x,y)$, for almost every $(x,y) \in [0,1]^{2}$. 

\end{enumerate} 

\end{corollary}

\begin{proof} 
The equivalence of (1) and (2) has already been shown in Corollary \ref{cor:deltatUW}. The equivalence of (2) and (3) follows from \cite[Theorem 3.17]{abraham2025probability} and the equivalence of the metrics $\delta_{\square, \textsf{LP}}$ and $\delta_\square$ established in Proposition \ref{ppn:UWlayersWp}.  
\end{proof}

\section{ Concluding Remarks and Future Directions }
\label{sec:directions}

In this paper, we develop the limit theory for sequences of dense multiplex networks. Specifically, we establish the equivalence between left-convergence and convergence in terms of the cut distance, present results on sampling convergence, identify useful decompositions of multiplexes and homomorphism densities, and discuss various examples. By focusing on multiplex networks, we have organized and simplified existing general results, with the aim of making them more accessible for applications in statistics and network science. The following is a partial list of possible future directions for further extending this theory.


\begin{itemize}

\item {\it Right convergence of multiplexes}: In this paper, our focus has been on left-convergence of multiplexes, that is, the convergence of homomorphism densities of fixed small multiplexes into increasingly large multiplexes. In graph limit theory, another complementary notion is right convergence, which deals with the convergence of homomorphism densities of large graphs into fixed small weighted graphs. These two notions of convergence are equivalent in the graphon setting, where they also correspond to the convergence of free energies and ground states of statistical physics models, such as Ising and Potts models \cite{graph_limits_II}. Right convergence also has interesting implications on parameter estimation in Ising models \cite{BM}. Recently, Zucal \cite{zucal2024probability} developed a right convergence theory for probability graphons, which, in particular, includes multiplex networks. It would be interesting to further explore how this framework connects to the convergence of free energies/ground states, and parameter estimation in statistical physics models on multiplex networks (see \cite{jang2015ashkin,krawiecki2018ferromagnetic,gomez2015layer} among others).

\item {\it Large deviations of random multiplexes}: Among the major applications of the graphon framework is the theory of large deviations for random graphs. This framework allows for the precise computation of tail probabilities of homomorphism densities and provides a way to understand the limiting structure of graphs conditioned on rare events \cite{CV,chatterjee2017large,chatterjee_pd,lubetzky2015replica}. Very recently, \cite{probabilitygraphons} developed the analogous theory of large deviations for probability graphons. In forthcoming work \cite{multiplex}, we study large deviations for random multiplexes, using the framework of multiplexons outlined in this paper.

\item {\it Fluctuations and moment-based network estimation}:  Another direction would be to derive fluctuations (limiting distributions) of homomorphism densities in a random multiplex. This has been established for graphon-based random graphs in \cite{bickel2011method,graphondistribution,bhattacharya2021fluctuations,feray2020graphons}. Obtaining analogous results for random multiplexes would enable extending moment-based statistical network estimation techniques \cite{zhang2022edgeworth,borgs2010moments,bickel2011method,graphondistribution} to the multiplex setting. In a related direction, very recently, \citet{triangles2025networks} obtained a multivariate Poisson approximation for triads in multislice Erd\H{o}s-R\'enyi networks.

\item {\it Multiplexes with given degree sequence}: 
Random graphs that are uniformly chosen with a given degree sequence are well-studied. 
In the dense regime, under mild conditions, such graphs converge to a limiting graphon \cite{degreesequence}. This result is also closely connected to the statistical properties of the $\beta$-model, a natural exponential family with the degree sequence as the sufficient statistic. For multiplex networks, various notions of degree can be considered, as discussed in Section \ref{sec:degreeclusteringcoefficient} (see also \cite[Section 6.2]{bianconi2018multilayer}). Fixing a notion of degree sequence, one can investigate the asymptotic structure of multiplex networks chosen uniformly with a given degree sequence, as well as statistical properties of the multiplex analogue of the $\beta$-model (see \cite[Section 10.3.5]{bianconi2018multilayer}).

\item {\it Limits of sparse multiplex networks}: The framework of graphons, and consequently that of multiplexons, can only capture the limits of dense networks. Several approaches to capturing limits of sparse graphs have been proposed in the literature. This includes notions of local weak convergence \cite{benjamini2001recurrence,localweakconvergence,BolRio09,hatami2014limits}), $L^p$-convergence \cite{graph_limits_sparse_I,graph_limits_sparse_II},  $s$-convergence or graphops \cite{kunszenti2019measures,backhausz2022action}, and graphexes \cite{borgs2021sparse,borgs2020identifiability,borgs2019sampling,caron2017sparse}, among several others. Extending these ideas to sparse multiplex networks is another future direction. 

\end{itemize}

\small{\noindent{\textbf{Acknowledgment:} B. B. Bhattacharya was supported by NSF CAREER grant DMS 2046393 and a Sloan Research Fellowship. }

\small 

\bibliographystyle{abbrvnat} 
\bibliography{references,bibliography}

\normalsize 

\appendix

\section{Proof of Proposition \ref{lem:compact}}
\label{sec:compactpf}

For any partition $\Pi = \{A_1,\dots,A_{N}\}$ of $[0,1]$ into $N \geq 1$ parts and any graphon $W \in \wsp$, define 
\begin{align}\label{eq:Wpartition}
W^{\Pi}(x,y) \defeq \sum_{1 \leq i,j\leq N} \bm 1\{ (x,y) \in A_i\times A_j \} \frac{1}{|A_i||A_j|}\int_{A_i\times A_j} W(x, y)\,\mathrm dx\, \mathrm dy. 
\end{align}
Now, fix any sequence $\{\bm W_n\}_{n \geq 1}$, with $\bm W_n  = ((W_n)_1, \ldots, (W_n)_\intW) \in \wsp^\intW$. Then for each $s \in [\intW]$ and $n, \kappa \geq 1$, as in the proof of \cite[Theorem 9.23]{Lov12}, select a partition $\Pi_{n, \kappa}\lyr{s}$ of $[0,1]$ such that the following hold: 

\begin{enumerate}[label=\textbf{(\arabic*)}]

\item \label{itm:propertyI} $\| (W_n)_{s} - (W_n)_{s}^{\Pi_{n,\kappa}\lyr{s}} \|_\square \leq \frac{1}{\kappa}$, where $(W_n)_{s}^{\Pi_{n,\kappa}\lyr{s}}$ is defined as in \eqref{eq:Wpartition} with $\Pi= \Pi_{n,\kappa}\lyr{s}$, 

\item  \label{itm:propertyII} the partition $\Pi_{n,\kappa+1}\lyr{s}$ refines $\Pi_{n,\kappa}\lyr{s}$,

\item \label{itm:propertyIII} $|\Pi_{n,\kappa}\lyr{s}| = c_{\kappa}$, where $c_{\kappa}$ depends only on $\kappa$.
\end{enumerate} 

Such a partition exists by \cite[Lemma 9.15]{Lov12}. Define 
\begin{align}\label{eq:WnK}
\bm{W}_{n,\kappa}  := \left( (W_n)_{1}^{\Pi_{n,\kappa}\lyr{1}} , (W_n)_{2}^{\Pi_{n,\kappa}\lyr{2}} , \ldots, (W_n)_{L}^{\Pi_{n,\kappa}\lyr\intW} \right) .
\end{align}
Also, let $$\Pi_{n,\kappa} = \left\{\bigcap_{s=1}^{\intW} A\lyr{s}: A\lyr{s} \in \Pi_{n,\kappa}\lyr{s}, \text{ for } s \in [\intW]\right\}.$$ Then $|\Pi_{n,\kappa}| \leq c_{\kappa}^{\intW}$ and $\Pi_{n,\kappa+1}$ is a refinement of $\Pi_{n,\kappa}$. Furthermore, since $\Pi_{n,\kappa}$ refines $\Pi_{n,\kappa}\lyr{s}$, for each $s \in [\intW]$, we may regard 
$$\left\{ (W_n)_{1}^{\Pi_{n,\kappa}\lyr{1}} , (W_n)_{2}^{\Pi_{n,\kappa}\lyr{2}} , \ldots, (W_n)_{L}^{\Pi_{n,\kappa}\lyr\intW} \right \}$$ as block graphons restricted to the partition $\Pi_{n,\kappa}$. Now, for each $n \geq 1$ consider a measure preserving bijection $\psi_n:[0, 1] \rightarrow [0,1]$ that maps each set in $\Pi_{n,\kappa}$ to an interval. 
Then 
from \eqref{eq:WnK} it follows that $$\bm{W}^{\psi_n}_{n,\kappa} = \left( \left((W_n)_1^{\psi_n}\right)^{\Lambda\lyr{1}_{n,\kappa}} , \left((W_n)_2^{\psi_n}\right)^{\Lambda\lyr{2}_{n,\kappa}}, \ldots, \left((W_n)_\intW^{\psi_n} \right)^{\Lambda\lyr\intW_{n,\kappa}} \right), $$ 
where, for $s \in [\intW]$, $$\Lambda\lyr{s}_{n,\kappa} = \{\psi_n^{-1}(A) : A \in \Pi_{n,\kappa}\lyr{s} \}$$ 
is the pushforward partition of $\Pi_{n,\kappa}\lyr{s}$ under the map $\psi_n$. Note that each element of $\Pi_{n,\kappa}\lyr{s}$ is a disjoint union of intervals. 
Also, recalling \eqref{eq:UWdeltalayer}, 
\begin{align}\label{eq:WKpartition}
\delta_{\square}({\bm{W}}_{n}, {\bm{W}}^{\psi_{n}}_{n,\kappa}) = \inf_{\sigma \in \mathcal{F}} \| \bm{W}_{n} - (\bm{W}^{\psi_{n}}_{n,\kappa})^{\sigma} \|_\square 
& = \inf_{\tau \in \mathcal{F}} \| \bm{W}_{n} - \bm{W}^{\tau}_{n,\kappa} \|_\square \nonumber \\ 
& \leq \| \bm{W}_{n} - \bm{W}_{n,\kappa} \|_\square \nonumber \\ 
& \leq \sum_{s=1}^{\intW} \| (W_{n})_s - (W_n)_s^{\Pi\lyr{s}_{n,\kappa}} \|_\square \leq \frac{L}{\kappa} , 
\end{align}
where the last inequality uses property \ref{itm:propertyI}. 

Now, by \cite[Claim 9.24]{Lov12}, there exists a subsequence $\{n_m\}_{m \geq 1}$ such that for each $\kappa \geq 1$ and for all $s \in [\intW]$, 
$$\lim_{m \rightarrow \infty}\left((W_{n_m})_s^{\psi_{n_m}}\right)^{\Lambda\lyr{s}_{n,\kappa}} (x, y) = U^{(\kappa)}_{s}(x, y) ,$$ 
for almost every $x, y \in [0,1]$ and some graphon $U^{(\kappa)}_{s} \in \wsp$, which is a step function with at most $c_{\kappa}$ steps (where each step is a finite union of intervals). 
Hence, setting $\bm{U}^{(\kappa)} \defeq ( U^{(\kappa)}_1, U^{(\kappa)}_2, \ldots, U^{(\kappa)}_\intW)$ and by the bounded convergence theorem,\footnote{For any multiplexon $\bm W = (W_1, W_2, \ldots, W_L) \in \wsp^{\intW}$, we define the $L_1$-norm as $\|\bm{W}\|_1 := \sum_{s=1}^\intW \|W_s\|_1=  \sum_{s=1}^\intW \int_{[0,1]^2} |W_{s}(x, y)| \mathrm d x \mathrm dy$. } 
\begin{align}\label{eq:WnKU}
\lim_{m \rightarrow \infty} \|\bm{W}^{\psi_{n_m}}_{n_m,\kappa} - \bm{U}^{(\kappa)}\|_1 
= 0, 
\end{align}  
for each $\kappa \geq 1$.

Next, let $\Pi_{\kappa}^{(s)}$ be the partition of $[0, 1]$ into the steps of $U^{(\kappa)}_{s}$, for $s \in [\intW]$. From the limiting analogue of property \ref{itm:propertyII}, we have 
\begin{align}\label{eq:partitionU}
U^{(\kappa)}_{s} = (U^{(\kappa+1)}_{s})^{\Pi_{\kappa}^{(s)}} , 
\end{align}
for $s \in [\intW]$. Now, suppose $(X, X') \sim \mathrm{Unif}([0,1]^{2})$. Then, \eqref{eq:partitionU} implies that the sequence $\{U^{(\kappa)}_{s}(X, X')\}_{\kappa \geq 1}$ is a bounded martingale, for each $s \in [\intW]$. Hence, by the martingale convergence theorem (see \cite[Theorem A.12]{Lov12}), the sequence $\{U^{(\kappa)}_{s}(X, X')\}_{\kappa \geq 1}$ converges almost surely to some random variable $U_{s}$, for each $s \in [\intW]$. Then setting $\bm U= (U_1, U_2, \ldots, U_L)$ and applying the bounded convergence theorem gives, 
\begin{align}\label{eq:UK}
\lim_{\kappa \rightarrow \infty}\| \bm{U}^{(\kappa)} - \bm{U} \|_1 = 0. 
\end{align}

Combining \eqref{eq:WKpartition}, \eqref{eq:WnKU}, and \eqref{eq:UK} gives, 
\begin{align*}
\delta_{\square}({\bm{W}}_{n_m}, {\bm{U}}) &  \leq \delta_{\square}({\bm{W}}_{n_m}, {\bm{W}}^{\psi_{n_m}}_{n_m,\kappa}) + \delta_{\square}( {\bm{W}}^{\psi_{n_m}}_{n_m,\kappa} , {\bm{U}}^{(\kappa)} ) + \delta_{\square} ( {\bm{U}}^{(\kappa)} , {\bm{U}} ) \nonumber \\ 
& \leq \frac{L}{\kappa} + \| \bm{W}^{\psi_{n_m}}_{n_m,\kappa} - \bm{U}^{(\kappa)} \|_1 + \|\bm{U}^{(\kappa)} - \bm{U} \|_1  \rightarrow 0,  
\end{align*}
as $m \to\infty$ followed by  $\kappa \to\infty$. This establishes the compactness of $\wsp_\square^\intW$. \hfill $\Box$ 

\section{ Proof of Proposition \ref{ppn:complete} } 
\label{sec:randomconvergence}

Fix $S \in \P_0([r])$. Then, from \cite[Corollary 3.4]{graph_limits_I}, for each $L \geq 2$, there is an equipartition $\Pi_{S}$ of $[0, 1]$ into $L$ measurable sets such that  
\begin{align*}
 \| \widebar{W}_S - \widebar{W}_S^{\Pi_S} \|_{\square}  \leq \frac{4}{\sqrt{\log L}}. 
\end{align*}  
Let $\Pi'$ be the common refinement of the partitions $\{ \Pi_{S} : S \in \P_0([r]) \}$. Note that $|\Pi'| \leq L^{R-1}$ and each set in $\Pi'$ has Lebsegue measure at most $\frac{1}{L}$. Also, 
\begin{align*}
\|  \widebar{W}_S - \widebar{W}_S^{\Pi'}  \|_{\square} \leq  \| \widebar{W}_S - \widebar{W}_S^{\Pi_S} \|_{\square}  \leq \frac{4}{\sqrt{\log L}}.
\end{align*}
Then, by \cite[Lemma 9.15]{Lov12}, there is a equipartition $\Pi$ of $[0, 1]$ into $L^R$ measurable sets that refines the partition $\Pi'$ such that 
\begin{align}\label{eq:WrefineL}
\|  \widebar{W}_S^{\Pi} - \widebar{W}_S \|_{\square} \leq  2 \|  \widebar{W}_S^{\Pi'} - \widebar{W}_S \|_{\square} + \frac{1}{L}  \leq \frac{9}{\sqrt{\log L}}.
\end{align} 
Therefore, denoting $\widebar{\bm{W}}^{\Pi} = (\widebar{W}_S^{\Pi} )_{ S \in \P_0([r]) }$ and adding both sides of \eqref{eq:WrefineL} over $S \in \P_0([r])$ gives, 
\begin{align}\label{eq:WL}
\| \widebar{\bm{W}}^{\Pi} - \widebar{\bm{W}} \|_\square \leq \frac{9 R}{\sqrt{\log L}} . 
\end{align}

Next, let $\bm{\eta} = \{\eta_1, \eta_2, \ldots, \eta_n\}$ be the i.i.d. $\mathrm{Unif}([0,1])$ random variables used to generate the random multiplex $\bm{G}(n, \widebar{\bm{W}})$. Denote by $\widebar{W}_{S}^{\Pi}[\bm{\eta}]$ the $n \times n$ matrix with $(i, j)$-th element $\widebar{W}_{S}^{\Pi}(\eta_i, \eta_j)$, for $i, j \in [n]$, and let $(\widebar{W}_{S}^{\Pi})^{(n)}$ be the associated graphon. Similarly, define $\widebar{W}_{S}[\bm{\eta}]$ to be the $n \times n$ matrix with $(i, j)$-th element 
$\widebar{W}_{S}(\eta_i, \eta_j)$, for $i, j \in [n]$, and let $(\widebar{W}_{S})^{(n)}$ be the associated graphon. 
Then by \cite[Lemma 9.15]{Lov12} and the proof of \cite[Lemma 10.16]{Lov12}, 
$$\E [ | \|(\widebar{W}_{S}^{\Pi})^{(n)} - \widebar{W}_{S}^{(n)} \|_\square - \|\widebar{W}^{\Pi}_{S} - \widebar{W}_S\|_\square | ] \leq 10 n^{-\frac{1}{4}} ,$$
for all $S \in \P_0([r])$ and $n$ large enough. Thus, setting $(\widebar{\bm{W}}^{\Pi})^{(n)} = ((\widebar{W}_{S}^{\Pi})^{(n)})_{S \in \P_0([r])}$, $\widebar{\bm{W}}^{(n)} = (\widebar{W}_{S}^{(n)})_{S \in \P_0([r])},$ and $L= \lceil n^{\frac{1}{4 R}} \rceil$ in \eqref{eq:WL} gives, 
\begin{align}\label{eq:WnWn}
\E [ \|(\widebar{\bm{W}}^{\Pi})^{(n)} - \widebar{\bm{W}}^{(n)} \|_\square] & \leq  \E [ | \|(\widebar{\bm{W}}^{\Pi})^{(n)} - \widebar{\bm{W}}^{(n)} \|_\square - \|\widebar{\bm{W}}^{\Pi} - \widebar{\bm{W}}\|_\square | ] + \| \widebar{\bm{W}}^{\Pi} - \widebar{\bm{W}} \|_\square \nonumber \\ 
& \leq \frac{19 R^{\frac{3}{2}}}{\sqrt{\log n}} , 
\end{align} 
for $n$ large enough. 

Next, consider the cumulative decomposition of the empirical graphon: $\bm{\widebar{W}}^{\bm{G}_{n}} = (\widebar{W}^{\bm{G}_{n}}_S)_{S \in \mathcal{P}_0([r])}$. Note that $\widebar{W}^{\bm{G}_{n}}_S$ is the empirical graphon corresponding to a $\widebar{W}_S$-random graph. Hence, by equation (10.9) of \cite{Lov12} (which is a consequence of \cite[Lemma 10.11]{Lov12}), 
\begin{align}\label{eq:WGnWn}
\E [ \| \bm{\widebar{W}}^{\bm{G}_{n}} - \widebar{\bm{W}}^{(n)} \|_{\square} ] = \sum_{S\in \P_0([r])} \E \left[  \| \widebar{W}^{\bm{G}_{n}}_S - \widebar{W}^{(n)}_{S} \|_{\square} \right] \leq \frac{11 R}{\sqrt{n} } .  
\end{align}

Finally, we consider the difference between $(\widebar{\bm{W}}^{\Pi})^{(n)}$ and $\widebar{\bm{W}}^{\Pi}$. For each $S \in \P_0([r])$, the functions $\widebar{W}^{\Pi}_{S}$ and $(\widebar{W}^{\Pi}_{S})^{(n)}$ are both block functions with equal number of steps with the same values in the corresponding steps. Denote the $s$-th step in $\Pi$ by $V_s$, for $1 \leq s \leq |\Pi|$. Since $\Pi$ is an equipartition, the measure of $V_s$, which is the $s$-th step in $\widebar{W}^{\Pi}_{S}$, is $\frac{1}{L^{R}}$. On the other hand, the measure of the $s$-th step in $(\widebar{W}_S^{\Pi})^{(n)}$ is $\frac{|V_s \cap \{\eta_1, \eta_2, \ldots, \eta_n\}|}{n}$. Then 
$$\delta_{\square}((\widebar{\bm{W}}^{\Pi})^{(n)} , \widebar{\bm{W}}^{\Pi}) = \inf_{\sigma \in \mathcal{F}} \sum_{S \in \P_0([r])} \| (\widebar{W}_S^{\Pi})^{(n)} - \widebar{W}_S^{\Pi} \|_{\square} \leq R \sum_{s=1}^{|\Pi|} |\Delta_s| , $$
where $\Delta_s := \frac{|V_s \cap \{\eta_1, \eta_2, \ldots, \eta_n\}|}{n} - \frac{1}{|\Pi|} $. Note that $\E[\Delta_1] = 0$ and $\E[\Delta_1^2] = \frac{1}{n |\Pi|}(1-\frac{1}{|\Pi|})$. Hence,  
\begin{align}\label{eq:WnW}
\E [ \delta_{\square}((\widebar{\bm{W}}^{\Pi})^{(n)} , \widebar{\bm{W}}^{\Pi}) ]  \leq R \sum_{s=1}^{|\Pi|} \E [|\Delta_s|] = R |\Pi| \E[ | \Delta_1 | ] \leq R |\Pi| \left( \E[\Delta_1^2] \right)^{\frac{1}{2}} & = R \sqrt{\frac{|\Pi|-1}{n}} \nonumber \\ 
& \leq R n^{-\frac{3}{8}} ,  
\end{align}  
since $|\Pi| = L^R \leq n^{\frac{1}{4}}$. 
%
%
Combining \eqref{eq:WL}, \eqref{eq:WnWn}, \eqref{eq:WGnWn}, and \eqref{eq:WnW}, the result in \eqref{eq:WGnW} we have 
$$\E[\delta_\square(\widebar{\bm{W}}^{\bm{G_n}}, \widebar{\bm W})] \leq \frac{ 30 R^{\frac{3}{2}}}{\sqrt{\log n}} , $$  
for $n$ large enough. Now, consider the function $f(G(n, \widebar{\bm{W}})) = \frac{n}{R} \delta_\square(\widebar{\bm{W}}^{\bm{G_n}}, \widebar{\bm W})$. This is a 1-Lipschitz function, hence by Azuma's inequality (similar to \cite[Theorem 10.3]{Lov12}), 
\begin{align*} 
\mathbb{P}\left(\delta_\square(\widebar{\bm{W}}^{\bm{G_n}}, \widebar{\bm W}) > \frac{32 R^{\frac{3}{2}}}{\sqrt{\log n}}\right) & \leq \mathbb{P}\left(\delta_\square(\widebar{\bm{W}}^{\bm{G_n}}, \widebar{\bm W})] -\E[\delta_\square(\widebar{\bm{W}}^{\bm{G_n}}, \widebar{\bm W})] > \frac{2 R^{\frac{3}{2}}}{\sqrt{\log n}}\right) \nonumber \\ 
& = \mathbb{P}\left( f(G(n, \widebar{\bm{W}})) - \E[f(G(n, \widebar{\bm{W}}))] > \frac{2 n \sqrt{R} }{\sqrt{\log n}}\right) \leq e^{- \frac{ 2 R n}{ \log n}} , \nonumber 
\end{align*} 
which completes the proof of Proposition \ref{ppn:complete}. \hfill $\Box$

\section{ Proof of Proposition \ref{ppn:UWlayersWp} } 
\label{sec:UWlayersmetricpf}

Since the L\'evy-Prokhorov metric is bounded above by the Total Variation ($\mathsf{TV}$) distance, we have, for any $A, B \subseteq [0, 1]$, 
\begin{align} 
\mathsf{LP} \left( \int_{A \times B}\bm{U}(x, y; \cdot ) , \int_{A \times B}\bm{W}(x, y; \cdot ) \right)  & \leq \mathsf{TV} \left( \int_{A \times B}\bm{U}(x, y; \cdot ) , \int_{A \times B}\bm{W}(x, y; \cdot ) \right) \nonumber \\ 
& = \frac{1}{2} \sum_{S \in \P([r])} \left|  \int_{A \times B} \left( \hat U_S(x, y) - \hat W_S(x, y) \right) \mathrm d x \mathrm d y \right| \tag*{(recall \eqref{eq:tHWprobabilitygraphon})} \nonumber \\ 
& \leq \frac{1}{2} \sum_{S \in \P([r])} \sup_{A, B \subseteq [0, 1]} \left|  \int_{A \times B} \left( \hat U_S(x, y) - \hat W_S(x, y) \right) \mathrm d x \mathrm d y \right| \nonumber \\ 
& \leq \frac{1}{2} \sum_{S \in \P([r])} \left\|  \hat U_S - \hat W_S \right\|_{\square} \nonumber \\ 
\label{eq:WpUWSlayers} & =  \frac{1}{2} \left\|  \hat{\bm{U}} - \hat{\bm W} \right\|_{\square} . 
\end{align}
Hence, taking supremum over $A,  B \in [0, 1]$ and infimum over $\sigma \in \mathcal{F}$ on both sides of \eqref{eq:WpUWSlayers} gives, $2 \delta_{\square, \mathsf{LP}}(\bm{U} , \bm {W}) \leq \delta_{\square} (\hat{\bm{U}}, \hat{\bm W} )$. This proves the lower bound on $\delta_{\square}$ in \eqref{eq:UWlayersWp}.

For the upper bound, denote by $C_{BL}(\bm{Z})$ the collection of all functions $F: \P([r]) \rightarrow \mathbb{R}$ which are 1-Lipschitz and $\max_{S \in \P([r])} |F(S)| \leq 1$. Then, recalling the definition of the Kantorovitch-Rubinstein (\textsf{KR}) metric, for any $S \in \P([r])$,  
\begin{align}\label{eq:ABUWS}
\left|\int_{A \times B} \left( \hat U_{S}(x, y) - \hat W_{S}(x, y) \right) \mathrm d x \mathrm d y  \right| & \leq  \sup_{F \in C_{BL}(\bm{Z})} \int_{A \times B} \left( \sum_{S' \in \P([r])} F(S') (\hat U_{S'}(x, y) - \hat W_{S'}(x, y))  \right) \mathrm d x \mathrm d y \nonumber \\ 
& = \mathsf{KR} \left( \int_{A \times B}\bm{U}(x, y; \cdot ) , \int_{A \times B}\bm{W}(x, y; \cdot ) \right) \nonumber \\ 
& \leq 3 \mathsf{LP} \left( \int_{A \times B}\bm{U}(x, y; \cdot ) , \int_{A \times B}\bm{W}(x, y; \cdot ) \right) , 
\end{align}
where the last step uses \cite[Theorem 3.22]{abraham2025probability}. Now, taking supremum over $A,  B \in [0, 1]$ and adding over $S \in \P([r])$ on both sides of \eqref{eq:ABUWS} gives, 
\begin{align}\label{eq:ABUWS}
\left\| \hat{\bm{U}} - \hat{\bm{W}} \right\|_{\square} & \leq 3 R \left\| \bm{U} - \bm{W} \right\|_{\square, \mathsf{LP}} . 
\end{align} 
Hence, taking the infimum over $\sigma \in \mathcal{F}$ on both sides yields the upper bound on $\delta_{\square}$ in \eqref{eq:UWlayersWp}. 
\hfill $\Box$

\end{document}